\documentclass[a4paper,12pt]{amsart}
\newcommand{\newc}{\newcommand}
\usepackage{amsmath,amsfonts,amsthm,tensor}
\newtheorem{thm}{Theorem}[section]
\newtheorem{defn}[thm]{Definition}
\newtheorem{prop}[thm]{Proposition}

\newtheorem{remk}[thm]{Remark}

\def\sideremark#1{\ifvmode\leavevmode\fi\vadjust{\vbox to0pt{\vss
 \hbox to 0pt{\hskip\hsize\hskip1em
 \vbox{\hsize2.5cm\tiny\raggedright\pretolerance10000
  \noindent #1\hfill}\hss}\vbox to8pt{\vfil}\vss}}}%
                        
\newc{\ds}{\displaystyle}
\newc{\nn}[1]{(\ref{#1})}
\newc{\nd}{\nabla}
\newc{\Sym}{\mbox{Sym}}
\newc{\CurlyD}{\mathfrak{D}}
\newc{\Action}{\mathcal{S}}
\newc{\GenTen}{T}                        
\newcommand{\cL}{\mathcal{L}}
\newc{\LagDen}{\cL_A}

\newcommand{\End}{\operatorname{End}}

\newc{\g}{\mbox{\textsl{g}}}             
\newc{\hatg}{\hat{\g}}
\newc{\hatgsup}{\hat{\mbox{\textsl{\tiny g}}}}    
\newc{\hatgsub}{\hat{\mbox{\textsl{\tiny g}}}}    
\newc{\gsup}{\mbox{\textsl{\tiny g}}}    
\newc{\gsub}{\mbox{\textsl{\tiny g}}}    
\newc{\bg}{\textbf{\g}}                  
\newc{\Ric}{\mbox{\textit{Ric}}}
\newc{\Sc}{\mbox{\textit{Sc}}}
\newc{\ObTen}{\mathcal{O}_{ab}}           
\newc{\ObTenSix}{\mathcal{O}_{ab}^{6}}    
\newc{\ObTenGenDim}{\mathcal{O}_{ab}^{n}} 
\newc{\ObTenSixac}{\mathcal{O}_{ac}^{6}}
\newc{\ObTenSixcd}{\mathcal{O}_{cd}^{6}}
\newc{\ObTenSixce}{\mathcal{O}_{ce}^{6}}
\newc{\ObTenSixic}{\mathcal{O}_{ic}^{6}}
\newc{\h}{h}                             
\newc{\FGI}{\mbox{\textit{FGI}}}         
\newc{\ScaleDen}{\xi^{\gsup}}             
\newc{\ScaleDenHat}{\xi^{\hatgsup}}
\newc{\TanB}{T\hspace{-0.2ex}M}          
\newc{\CotB}{T^*\hspace{-0.4ex}M}        
\newc{\GenBund}{V}                       
\newc{\AnyBund}{W}                       
\newc{\adGenBund}{ad\GenBund}            
\newc{\TBund}{\mathcal{T}}               
\newc{\TBundB}{\mathcal{T}^B}            
\newc{\ce}{\mathcal{E}}                  
\newc{\dbun}{\ce}                        
\newc{\MetricBund}{\mathcal{Q}}          
\newc{\TenBundrs}{T^{(r,s)}}               
\newc{\TenBundOneZero}{T^{(1,0)}}          
\newc{\TenBundZeroOne}{T^{(0,1)}}         
\newc{\TenBundZerok}{T^{(0,k)}}            
\newc{\AngleIProd}{\langle\cdot,\cdot\rangle}
\newc{\BilinearMap}{\langle\cdot,\cdot\rangle}
\newc{\Integers}{\mathbb{Z}}             
\newc{\Reals}{\mathbb{R}}                
\newc{\ConnectionSet}{\mathcal{A}}
\newc{\cc}{\mbox{\slshape\bfseries c}}   
\begin{document}

\title{A conformally invariant Yang-Mills type energy and
  equation on 6-manifolds\vspace{3.0ex}}

\author{A. Rod Gover, Lawrence J.\ Peterson, and Callum Sleigh}

\address{A.R.G.: Department of Mathematics\\
  The University of Auckland\\
  Private Bag 92019\\
  Auckland 1142\\
  New Zealand}
\email{r.gover@auckland.ac.nz}
\address{L.J.P.:  Department of Mathematics\\
  The University of North Dakota\\
  101 Cornell Street Stop 8376\\
  Grand Forks, ND 58202-8376\\
  USA}\vspace{1.0ex}
\email{lawrence.peterson@und.edu\vspace{1.8ex}}
\email{callumsleigh@gmail.com}
\begin{abstract}
We define a conformally invariant action $\Action$ on gauge
connections on a closed pseudo-Riemannian manifold $M$ of
dimension~6. At leading order this is quadratic in the gauge
connection.  The Euler-Lagrange equations of $\Action$, with respect
to variation of the gauge connection, provide a higher-order
conformally invariant analogue of the (source-free) Yang-Mills
equations.

For any gauge connection $A$ on $M$, we define $\Action(A)$ by first
defining a Lagrangian density associated to $A$. This is not
conformally invariant but has a conformal transformation analogous to a
$Q$-curvature. Integrating this density provides the conformally
invariant action.

In the special case that we apply $\Action$ to the conformal
Cartan-tractor connection, the functional gradient recovers the
natural conformal curvature invariant called the Fefferman-Graham
obstruction tensor. So in this case the Euler-Lagrange equations are
exactly the ``obstruction-flat'' condition for 6-manifolds. This
extends known results for 4-dimensional pseudo-Riemannian manifolds
where the Bach tensor is recovered in the Yang-Mills equations of the
Cartan-tractor connection.
\end{abstract}

\thanks{2020 \textit{Mathematics Subject Classification}.  primary:
  53C18, 53C07; secondary: 58E15, 81T13, 53A55.\\ \indent
  A.R.G.\ gratefully acknowledges support from the Royal Society of
  New Zealand via Marsden Grants 16-UOA-051 and
  19-UOA-008. C.S.\ Acknowledges support of an FRDF Postdoctoral
  Fellowship at the University of Auckland.}

\maketitle

\pagestyle{myheadings} \markboth{Gover, Peterson, and Sleigh}{A
  Conformally Invariant Yang-Mills Type Energy}

\section{Introduction}\label{IntroSect}

Central to both the standard model in physics and our understanding of
differentiable 4-manifolds are the extremely well-known gauge field
equations known as the Yang-Mills equations
\cite{Donaldson,DonaldsonKron,tHooft,YM}. A feature of these
equations, that is crucially important in both spheres, is their
conformal invariance in dimension~4 \cite{Donaldson,FP1,FP,JR,Taubes}.  On
the background of a Riemannian or pseudo-Riemannian 4-manifold
$(M,\g)$, consider a gauge connection $A$ with curvature $F_{ab}$
(where the gauge field indices are suppressed and the tensor indices
are abstract).  The Yang-Mills Lagrangian density is (up to a constant
factor we shall ignore)
$\LagDen:=\operatorname{Trace}(\g^{ab}\g^{cd}F_{ac}\circ F_{bd})$.
For any $\Upsilon\in C^{\infty}(M)$, let
$\hatg_{ab}=e^{2\Upsilon}\g_{ab}$.  Under the conformal transformation
$\g_{ab}\mapsto\hatg_{ab}$, the Lagrangian transforms to
$e^{-4\Upsilon}\LagDen$, and this exactly balances the corresponding
transformation of the metric measure $d\mu^{\gsup}\mapsto
d\mu^{\hatgsup}= e^{4\Upsilon}d\mu^{\gsup}$.  The action $\Action$
given by
\begin{equation}\label{act1}
\Action(A):=\int_M \LagDen\,d\mu^{\gsup} ,
\end{equation}
for compactly supported $F$, is thus conformally invariant. It follows
at once that the functional gradient, with respect to variation of
connection, is also conformally invariant.  The Euler-Lagrange
equations that these determine, that is the (source-free) Yang-Mills
equations
\begin{equation}\label{ymeq}
\delta_A F=0,
\end{equation}
are thus conformally invariant as well.  Here $\delta_A $ is the
formal adjoint of the connection-twisted exterior derivative $d_A$.

This argument for conformal invariance evidently fails in dimensions
other than 4, and it is well-known, and easily verified, that the
equations (\ref{ymeq}) are not conformally invariant in other
dimensions \cite{Branson-ymc}.  It is thus natural to ask whether, in
dimensions $n$ other than 4, there is a replacement action/energy and
corresponding natural Euler Lagrange PDE system that is conformally
invariant. That is, on a smooth $n$-manifold we seek a map from pairs
$(\g,A)$, where $\g$ is a pseudo-Riemannian metric of any given
signature $(p,q)$, and $A$ is a principal connection, to corresponding
Lagrangian densities $\LagDen$ of conformal weight $-n$ so that the
expression \nn{act1} is again conformally invariant.  (We will discuss
densities and weights in Section~\ref{ConformalDen} below.)  There is
an important question here of whether the conformally invariant
4-dimensional Yang-Mills equations are an isolated phenomenon, or
whether rather they are the dimension 4 case of more general
picture. From elementary conformal geometry one does not expect such a
possibility in odd dimensions.

These questions are interesting from the perspective of gauge theory
in general, but they also have an interesting and direct motivation in
conformal geometry linked to the Bach tensor and its generalisations.
On pseudo-Riemannian 4-manifolds, the {\em Bach tensor} is an
important conformal invariant given by
\begin{equation}\label{Bach}
B_{ab}:= \nd^cA_{acb}+P^{cd}C_{cadb}.
\end{equation}
Here and below, $A_{acb}$, $P^{cd}$, and $C_{cadb}$ are the Cotton,
Schouten, and Weyl tensors, respectively, and $\nd$ is the Levi-Civita
connection. See Section \ref{ConformalGeom}. In other dimensions $n\geq 3$ we will also say that
$B_{ab}$, as given in \nn{Bach}, is the Bach tensor.  Metrics $\g$
that are conformal to Einstein metrics are Bach-flat (meaning
$B^{\gsup}_{ab}=0$).  So in dimension 4 the equations $B^{\gsup}_{ab}=0$ provide, in a
sense, a conformally invariant weakening of the Einstein equations, an
idea that dates back to the original work of Bach for a ``conformal
gravity'' theory \cite{Bach21} and that is still exploited
\cite{AO,CLW,GV,H,Mald}. Self-dual 4-manifolds are also Bach-flat and
so there is considerable interest in determining to what extent 
 there are other examples on closed Riemannian manifolds
\cite{CGGGR,LeB}.  The tensor \nn{Bach} is not conformally invariant
in dimensions other than 4, but in higher even dimensions there is a
conformally invariant higher-order analogue of \nn{Bach} that takes the form
\[
\ObTen=\Delta^{n/2-2}\nabla^c\nabla^d C_{acbd}+
\mbox{lower-order terms}.
\]
Fefferman and Graham discovered the tensor $\ObTen$ as an obstruction
to their ``ambient metric'' construction.  See \cite{FG2}.  For this
reason, $\ObTen$ is often referred to as the {\em Fefferman-Graham}
obstruction tensor.  In dimension~4, the tensors $\ObTen$ and $B_{ab}$
are equal.  We will often write $\ObTenGenDim$ instead of $\ObTen$ to
emphasise the dimension.

The Bach tensor also arises as the functional gradient (now with
respect to metric variations and where we are ignoring overall
constant factors) of the integral of the square of the conformal Weyl
tensor $|C|^2_{\gsub}=C^{abcd}C_{abcd}$.  (Here we use the metric to
raise and lower indices.)  As in the case of $\LagDen$, the weights
conspire so that $\int_M |C|^2_{\gsub}\,d\mu^{\gsup}$ is conformally
invariant in dimension 4. On closed manifolds, an alternative
Lagrangian density is provided by Branson's much studied $Q$-curvature
$Q^{\gsup}$ of \cite{BrO,BrFunctDet,Brsharp}. This alternative density
generalises, in that in all even dimensions there is such a
$Q$-curvature, and the Fefferman-Graham obstruction tensor is the
functional gradient of
\begin{equation}\label{IntegralQ}
\int_{M} Q^{\gsup}\,d\mu^{\gsup}
\end{equation}
with respect to metric variations \cite{Graham-Hirachi}.  The
situation is slightly more subtle now, however, because the
$Q$-curvature is not itself conformally invariant (whether we treat it
as a function or as a conformal density).  Under conformal
transformation, $Q$ changes by a term in which a linear operator acts
on the logarithm of the conformal factor.  This term is of divergence
form, and for this reason, \nn{IntegralQ} is conformally invariant
\cite{Brsharp}.

On a general Riemannian or pseudo-Riemannian manifold there is no
natural conformally invariant connection on the tangent bundle, but in
dimensions $n\geq 3$ there is a natural conformally invariant
connection on a related bundle of rank $n+2$. This is the conformal
tractor connection \cite{Thomas} due, in its original form, to Cartan
and Thomas \cite{Cart,TYThomas}.  See Section~\ref{TBundSect} below.
A nice feature of this connection is that in dimension 4, its
associated Yang-Mills equations recover exactly the Bach-flat
condition $B_{ab}=0$ \cite{Merk,GoverSombergSoucek}: Lemma 4.2 in
\cite{GoverSombergSoucek} states that in dimension 4, in an informal
tractor notation,
\begin{equation}\label{YM-Bach}
\delta_{\nd^{\TBund}}F_{\nd^{\TBund}} = c\cdot (0,0,0, B_{ab}),
\end{equation}
where $\nd^{\TBund}$ is the conformal tractor connection, and $c$ is
an (explicit non-zero) constant. One might hope that a
higher-dimensional conformally invariant Yang-Mills type theory would
generalise this picture.

In the current article, we treat these questions in dimension
$n=6$. It seems to us that the results should assist with extending
many of the studies and directions mentioned above from the confines
of dimension 4 to higher even dimensions.
For convenience, we will always view the gauge connection $A$ as
a linear connection acting on some vector bundle $\GenBund$, and we
let $F_A$ denote the curvature of $A$.  Our work will rely on the
operator $Q_2^A$ defined in Section~\ref{QTwoSect}.  This operator
belongs to a family of operators defined in
\cite{BransonGoverGen,BransonGover}.  In what follows, we will also
work with a bilinear mapping $\left<\cdot,\cdot\right>$ which we define
in Section~\ref{ConnectionSect}.

For closed (i.e.\ compact without boundary) 6-manifolds, we provide
the sought map from pairs $(\g,A)$, consisting of a metric and a
connection, to Lagrangian densities $\LagDen$ so that the action
$\Action$ given in \nn{act1} is conformally invariant.  The Lagrangian
density we provide, namely $\left< F_{A}, Q_{2}^{A} F_{A} \right>$, is
an analogue of the $Q$-curvature.  Here we mean that this density is
not conformally invariant but rather transforms, under conformal
rescaling of the metric $\g$ on $M$, by a linear divergence-type
operator acting on the logarithm of the conformal factor. This is an
immediate consequence of Proposition~\ref{Q2rescaling} below, which
was first derived in \cite{BransonGover}.  If $M$ is closed, it
follows that
\begin{equation}\label{ourA}
  \Action(A)=\int_M\left<F_{A},Q_{2}^{A}F_{A}\right>\,d\mu^{\gsup}
\end{equation}
  is
invariant under conformal change of the metric on $M$.  See
Proposition~\ref{SACoInv}.

An important feature of the Lagrangian density
$\left<F_A,Q_2^AF_A\right>$ is that it is quadratic in the jets of the
connection $A$ at leading order.  So if we consider the functional
gradient of $\Action$, with respect to variation of connection, we find
that this functional gradient is linear at leading order.  The
functional gradient therefore provides a
conformally invariant map $\CurlyD$ from connections $A$ to 1-forms
taking values in $\End(\GenBund)$:
$$
A\mapsto\CurlyD A.
$$
We thus obtain a higher-order conformally invariant analogue of the
Yang-Mills equations.  See Proposition~\ref{EulerLagrange} and
Definition~\ref{DDef}. By construction these equations are conformally
invariant, but a direct verification is provided in Proposition
\ref{conformalinv}.

As an application and test of this, in Section
\ref{ApplicationsSection} we show that the non-linear operator
$\CurlyD$ applied to the conformal tractor connection exactly recovers
the Fefferman-Graham obstruction tensor of conformal 6-manifolds. Thus
we obtain a perfect parallel of the situation in dimension 4, as
discussed above, and at the same time a new perspective on the
Fefferman-Graham obstruction tensor. See Theorem
\ref{ObstructionTheorem} which, again in an informal tractor notation,
states that
\begin{equation}\label{a-key}
\CurlyD (\nd^{\TBund})= k\cdot(0,0,0, \ObTen ),
\end{equation}
where $\nd^{\TBund}$ is the conformal tractor connection, and $k$
is an (explicit non-zero) constant. Compare this to (\ref{YM-Bach})
above.

It is not difficult to see that there are other conformally invariant
actions on connections in dimension 6. For example, apart from
divergences, we can add cubic terms
$\operatorname{Trace}(\g^{af}\g^{bc}\g^{de}F_{ab}\circ F_{cd}\circ
F_{ef})$ to the Lagrangian density.  However Theorem
\ref{ObstructionTheorem} (i.e.\ \nn{a-key}) suggests that \nn{ourA} is
distinguished.

Finally, in Section \ref{FGIRemark} we discuss the link between our
action \nn{SADef}, when specialised to the case that $A$ is the
conformal tractor connection, and a conformally invariant action we
construct using the Fefferman-Graham invariant of Proposition~3.4 of
\cite{FG2}.

It seems that recently there has been some interest in ``higher order
Yang-Mills'' flows \cite{K,Z}.  In these works, the generation of the
equations emphasises the use of broadly analogous higher-order actions
on Riemannian manifolds, rather than attention to conformal
properties, but it is possible that some of the results could be
applied to the equations we develop here. In \cite{K} Kelleher asks
the question of whether there might be conformally invariant
higher-order Yang-Mills type equations, and suggests that, if so,
conformal invariance might be a good distinguishing property. In
\cite{Baez} the author looks at a notion of Yang-Mills for higher
dimensions that is not closely related to the programme here.

%
%
\section{Notation and background}\label{Notation}

Throughout our work, $M$ will be a smooth manifold of dimension $n$.
Our main interest lies in smooth manifolds of dimension $6$, but in
this section, we work with general manifolds in general dimensions
$n$.  We let $\TanB$ and $\CotB$ denote the tangent and cotangent
bundles, respectively, of $M$.  We will often use Penrose's abstract
index notation.  Lowercase indices $a$, $b$, $c$, and so forth will be
associated to $\TanB$ and $\CotB$, and uppercase indices such as $B$,
$C$, and $D$ will be associated to other finite-dimensional vector
bundles.  Parentheses will indicate symmetrisation of indices.  For
example, if $T_{abc}$ is a rank~3 tensor, then
$T_{a(bc)}=(1/2)(T_{abc}+T_{acb})$.  We will always assume that
$\g_{ab}$ (or just $\g$) denotes a pseudo-Riemannian metric on $M$ and
that $R_{ab}{}^{c}{}_{d}$ is the Riemannian curvature tensor of the
Levi-Civita connection $\nd$ associated to $\g_{ab}$.  We use the sign
convention for $R_{ab}{}^{c}{}_d$ such that
\begin{equation}\label{RiemannConvention}
(\nd_a\nd_b-\nd_b\nd_a)v^c=R_{ab}{}^c{}_d v^d
\end{equation}
for all vector fields $v^c$ on $M$.  Let $k\in\Integers_{>0}$ and a
smooth tensor field $\GenTen^{c_1\ldots c_k}$ on $M$ be given.  Then
\begin{equation}\label{CommuteCovD}
  \nd_a\nd_b\GenTen^{c_1\ldots c_k}
  =
  \nd_b\nd_a\GenTen^{c_1\ldots c_k}
  +R_{ab}{}^{c_1}{}_i\GenTen^{i\ldots c_k}
  +\cdots
  +R_{ab}{}^{c_k}{}_i\GenTen^{c_1\ldots i},
\end{equation}
by \nn{RiemannConvention}.  We may use the metric $\g_{ab}$ to raise
or lower any of the free indices in \nn{CommuteCovD}.  We let
$\Delta :=\nd_a\nd^a$.  For any vector bundle $\AnyBund$ over $M$, let
$\Gamma(\AnyBund)$ denote the space of all smooth sections of
$\AnyBund$.  The notation $\AnyBund^*$ will always denote the dual of $\AnyBund$.
Our scaling convention for the wedge product of differential forms
will be such that for all $p,\,q\in\Integers_{\geq 0}$, all $x\in M$,
and all $p$-forms $\omega$ and $q$-forms $\psi$ at $x$,
\begin{equation}\label{WedgeConvention}
\omega\wedge\psi=\frac{(p+q)!}{p!q!}\,\mbox{Alt}(\omega\otimes\psi).
\end{equation}
%
%
\subsection{Conformal geometry}\label{ConformalGeom}
Let a manifold $M$ and pseudo-Riemannian metrics $\g$ and $\hatg$ on
$M$ be given, and suppose there is an $\Upsilon\in C^{\infty}(M)$ such
that $\hatg=e^{2\Upsilon}\g$.  Then we say that $\hatg$ is
\textit{conformal} to $\g$.  Let $\cc$ denote the set of all
pseudo-Riemannian metrics on $M$ which are conformal to $\g$.  We say
that $\cc$ is a \textit{conformal structure} on $M$ or a
\textit{conformal class} on $M$, and we say that the pair $(M,\cc)$ is
a \textit{conformal manifold}.  Throughout our work, $\hatg$ will
always denote the metric $e^{2\Upsilon}\g$ (for some $\g$ and
$\Upsilon$).  For any object, such as a tensor, that depends on the
choice of the metric in $\cc$, a hat\ \ $\widehat{}$\ \ will indicate
the object as determined by the metric $\hatg$.

We will need to introduce several classical tensors that occur
naturally in conformal geometry due to their simple transformation
laws under conformal rescaling.  The first tensor we will need is the
symmetric tensor $P_{ab}\in \Gamma (S^{2}\CotB)$, called the
\emph{Schouten tensor}, defined in general dimensions $n\geq 3$ by
\[
P_{ab}:= \frac{1}{n-2} \left(\Ric_{ab} -
\frac{1}{2(n-1)}\,\g_{ab}\,\Sc\right).
\]
Here $\Ric_{ab}=R^c{}_{acb}$ is the Ricci tensor, and
$\Sc=\Ric_{a}{}^a$ is the scalar curvature. The tensors $P$, $R$,
$\Ric$, and $\Sc$ depend on the choice of metric $\g$, and we may
write $P^{\gsup}$, $R^{\gsup}$, and so forth to indicate this dependence.
We will usually omit $\g$ from our notation, however.

By using the metric to raise indices, we may associate $P_{ab}$ to an
element of $\Gamma(\End(\CotB))$.  We will let the symbol $P \#$ denote
the corresponding derivation on general tensor fields. For example,
the action of $P \#$ on a two-form $\omega_{ab}$ is
\[
(P \#\omega)_{ab} = \tensor{P}{_a^c} \tensor{\omega}{_c_b} +
\tensor{P}{_b^{c}} \tensor{\omega}{_a_c}.
\]
We will use this later.  For all $r,s\in\Integers_{\geq 0}$, let
$\TenBundrs$ denote the bundle of tensors of contravariant rank
$r$ and covariant rank $s$.
For any vector bundle $\AnyBund$, we
may extend the action of $P\#$ (trivially) to sections of $\TenBundrs
\otimes\AnyBund$ as follows.  First, let
\[
P\#(\omega\otimes U) =
(P\#\omega)\otimes U
\]
for all $\omega\in\Gamma(\TenBundrs)$ and all
$U\in\Gamma(\AnyBund)$.  Then extend the action $P\#$ linearly to sums
of sections of the form $\omega\otimes U$, where
$\omega\in\Gamma(\TenBundrs)$ and $U\in\Gamma(\AnyBund)$.

Our work will also involve the scalar field $J$ given by
$J=P_a{}^a=P_{ab}\g^{ab}$.  From the second Bianchi identity, it
follows that
\begin{equation}\label{TraceNbP}
\nd_aP^{a}{}_b=\nd_bJ.
\end{equation}
The \textit{Cotton tensor} $A_{abc}$ is defined by
\begin{equation}\label{CottonTensor}
A_{abc}:=\nd_bP_{ca}-\nd_cP_{ba}.
\end{equation}
The Cotton tensor is trace-free.  The Weyl tensor is given
by
\begin{equation}\label{WeylTensor}
C_{abcd}:=R_{abcd}+\g_{cb}P_{ad}-\g_{ca}P_{bd}+\g_{da}P_{bc}-\g_{db}P_{ac}.
\end{equation}
In all dimensions $n\geq 3$, the Bach tensor $B_{ab}$ is given by
\nn{Bach}.  We note here that $B_{ab}$ is trace-free.  A short
 computation shows that $B_{ab}$ is symmetric.
%
%
\subsection{Conformal density bundles}\label{ConformalDen}
One may simplify many of the calculations and fomulae of 
conformal geometry by using the language of weighted density bundles.
By its definition, a conformal structure $\cc$ determines a ray
sub-bundle $\MetricBund$ of $S^{2}\CotB$. This sub-bundle is actually
a principal $\mathbb{R}^{+}$-bundle $\pi: \MetricBund\rightarrow M$.
For every $\g\in\cc$ and every $x\in M$, the fibre of $\MetricBund$
over $x$ consists of all metrics at $x$ which are conformal to $\g$ at
$x$.  For each $w\in \mathbb{R}$ we will denote by $\dbun[w]$ the
induced line bundle arising from the representation of
$\mathbb{R}^{+}$ on $\mathbb{R}$ given by $t \mapsto t^{-w/2}$. We
note that each bundle $\dbun[w]$ is canonically oriented.  Sections of
$\dbun[w]$ are called \emph{conformal densities of weight $w$} and
maybe identified (via the associated bundle construction) with
homogeneous functions of degree $w$ on $\MetricBund$.

There is a tautological function $\bg$ on $\MetricBund$ taking values
in $S^{2}\CotB$ which simply assigns to a point
$(\g_{x},x) \in\MetricBund$
the metric $\g_{x}$ at $x$. This determines a section
$\bg$ (or $\bg_{ab}$) of $S^{2}\CotB\otimes\dbun[2]$ called the
\emph{conformal metric}.  For any nonvanishing $\sigma \in
\Gamma(\dbun[1])$, it follows $\sigma^{-2}\bg_{ab}$ is a metric in the
conformal class, and we use the term \emph{conformal scale} for
$\sigma$.  In a similar fashion, we may define a section $\bg^{ab}$ of
$S^2\TanB\otimes\ce[-2]$ in such a way that $\bg^{ab}=\bg^{-1}$.  A
given pseudo-Riemannian metric $\g$ determines a positive section
$\ScaleDen$ of $\dbun[1]$ via the relation
$\g_{ab}=(\ScaleDen)^{-2}\bg_{ab}$.  We say that $\ScaleDen$ is the
\textit{scale density} associated to $\g$.  For any metric $\g$, the
Levi-Civita connection acts on sections of $\dbun[w]$ as follows.  Let
$f\in C^{\infty}(M)$ be given.  Then
$\nd((\ScaleDen)^wf)=(\ScaleDen)^wdf$, where $d$ is the exterior
derivative on functions.  Note that $\nd\ScaleDen=0$.

Note that $\bg_{ab}$ and $\bg^{ab}$ are conformally invariant, so we
may use $\bg^{ab}$ and $\bg_{ab}$ to raise and lower tensor indices in
a conformally invariant way.  This fits with the notion of weighted
tensors.  From this point forward, all tensors will implicitly be the
product of a tensor and a density of some weight $w$.  We say that the
tensor is \textit{weighted}, or that it \textit{carries} a weight.  The
Riemannian curvature tensor $R_{ab}{}^c{}_d$ will carry a weight of 0.
On the other hand, $\bg_{ab}$ will carry a weight of 2, and $\bg^{ab}$
will carry a weight of $-2$.  If we raise or lower an index, we will
always use the conformal metric to do this.  Thus $R_{abcd}$ will have
weight $2$, and $R_{ab}{}^{cd}$ will carry a weight of $-2$.  The
tensors $P_{ab}$, $J$, $A_{abc}$, $C_{abcd}$, and $B_{ab}$ will carry
weights $0$, $-2$, $0$, $2$, and $-2$, respectively.  Elements and
sections of vector bundles will also carry weights.

To work with weighted tensors and vectors efficiently, we introduce
additional notations.  For every $k\in\Integers_{\geq 0}$, let
$\Lambda^k(M)$ denote the bundle of $k$-forms on $M$, and for every
$w\in\Reals$, let $\Lambda^k(M,w)$ denote
$\Lambda^k(M)\otimes\dbun[w]$.  Let $\Omega^k(M,w)$ denote
$\Gamma(\Lambda^k(M,w))$.  Similarly, for any vector bundle $\AnyBund$
over $M$, let $\Lambda^k(\AnyBund)$ denote the bundle of $k$-forms on
$M$ with values in $\AnyBund$, and let $\Lambda^k(\AnyBund,w)$ denote
$\Lambda^k(\AnyBund)\otimes\dbun[w]$.  Let $\Omega^k(\AnyBund,w)$
denote $\Gamma(\Lambda^k(\AnyBund,w))$.  Note that
$\Lambda^k(M)=\Lambda^k(M,0)$.  Similar remarks apply to
$\Lambda^k(\AnyBund)$, $\Omega^k(M)$, and $\Omega^k(\AnyBund)$.

Remembering that tensors carry weights, we now state the
transformation rules for $P_{ab}$ and $J$ under conformal change of
metric.  In general dimensions $n$, the Schouten tensor transforms
according to the rule
\[
\widehat{P}_{ab} = P_{ab} - \nabla_{a} \Upsilon_{b} + \Upsilon_{a}
\Upsilon_{b} -\frac{1}{2}\bg_{ab} \Upsilon_{c} \Upsilon^{c}.
\]
Here $\nabla$ is the Levi-Civita connection associated to the original
metric $\g$, and (for example) $\Upsilon_a=\nd_a\Upsilon$.  Also,
$\Upsilon^c=\bg^{cd}\Upsilon_d$.  We will use these notational
conventions throughout our work.  Note that $\Upsilon^c$ carries a
weight of $-2$.  In general dimensions $n$, the conformal transformation
rule for $J$ is as follows:
\[
  \widehat{J}
  =
  J
  -\nd_a\Upsilon^a
  +\left(1-\frac{n}{2}\right)\Upsilon_a\Upsilon^a.
\]
%

%
%
\subsection{Integration of densities}\label{IntegrationSection}
We now consider integration of densities.  Let a conformal manifold
$(M,\cc)$, a section $\phi$ of $\ce[-n]$, which is of compact support,
and $\g\in\cc$ be given.  Then $\phi=(\ScaleDen)^{-n} f$ for some
$f\in C^{\infty}(M)$.
Let $d\mu^{\gsup}$ denote the pseudo-Riemannian measure
on $M$ associated to $\g$, and define $\int_M\phi\,d\mu$ by
\begin{equation}\label{IntDefnC}
\int_M\phi\,d\mu:=\int_Mf\,d\mu^{\gsup}.
\end{equation}
This integral is independent of the choice of the $\g\in\cc$ that we
use in its definition.  To see this, begin by recalling our convention
that $\hatg=e^{2\Upsilon}\g$.
One can show that
$\ScaleDen=e^{\Upsilon}\ScaleDenHat$.
Thus $(\ScaleDen)^{-n}=e^{-n\Upsilon}(\ScaleDenHat)^{-n}$, and hence
with $f$ as above,
\[
\phi=(\ScaleDen)^{-n}f=e^{-n\Upsilon}(\ScaleDenHat)^{-n}f=
(\ScaleDenHat)^{-n}e^{-n\Upsilon}f.
\]
So if we
use $\hatg$ to define $\int_M\phi\,d\mu$, we find that
\[
 \int_M\phi\,d\mu
=\int_Me^{-n\Upsilon}f\,d\mu^{\hatgsup}
=\int_Me^{-n\Upsilon}fe^{n\Upsilon}\,d\mu^{\gsup}
=\int_Mf\,d\mu^{\gsup}.
\]
This is consistent with \nn{IntDefnC}.

%
%
\section{Bundles and connections}\label{ConnectionSect}

In our discussions of connections and curvature on general vector
bundles, our conventions and basic definitions will essentially follow
those given in standard references such as \cite{DonaldsonKron}.  In
this section, we describe our conventions and define some of the
vector bundles and operators that our work will use.
We also state some conformal transformation rules for connections and
other operators.  For consistency with work in other contexts, we work
in general dimensions $n$ unless we indicate otherwise.

Let a finite-dimensional vector bundle $\AnyBund$ over $M$, a
connection $A$ on $\AnyBund$, and a nonnegative integer $k$ be given.
Let $\nd$ denote the standard coupling of $A$ and the Levi-Civita
connection.  When $\nd$ acts on weighted tensors associated to $\TanB$
or $\CotB$, it acts as the Levi-Civita connection.  If $\nd$ acts on a
section of $\AnyBund$, then it acts as $A$.  By linearity and the
Leibniz property, $\nd$ is extended to weighted sections of tensor
products of these bundles.  The symbol $\nd$ will denote a coupled
connection throughout much of our work.

Let $w\in\Reals$ be given.  The connection $A$ induces a twisted
exterior derivative
$$d_{A}:
\Omega^{k} (\AnyBund,w) \rightarrow \Omega^{k+1}(\AnyBund,w).$$
By \nn{WedgeConvention},
\begin{equation}\label{dBundle}
(d_{A}\omega)_{a_1\ldots a_{k+1}}{}^{B}
=
\frac{1}{k!}\sum_{\sigma\in S_{k+1}}(\mbox{sgn\ }\sigma)
\nd_{a_{\sigma(1)}}\omega_{a_{\sigma(2)}\ldots a_{\sigma(k+1)}}{}^{B}
\end{equation}
for all $\omega\in\Omega^{k}(\AnyBund,w)$.
Here $S_{k+1}$ denotes the set of all permutations of $\{1,\ldots,k+1\}$.

To give a formula for the formal adjoint of $d_{A}$, we begin by
letting a metric $h$ (of some signature) on $\AnyBund$ be given, and
we suppose that $h$ is preserved by the connection $A$.  Also let
$k\in\Integers_{\geq 0}$ and $w_1$, $w_2\in\Reals$ be given, and
suppose that $w_1+w_2=2k-n$.  We will use the conformal metric $\bg$
on $M$ together with $h$ to define a bilinear mapping
\[
\BilinearMap:
\Lambda^k(\AnyBund,w_1)\times\Lambda^k(\AnyBund,w_2)\rightarrow
\dbun[-n].
\]
Specifically, for all $x\in M$ and all
$\omega\in\Lambda^k(\AnyBund,w_1)$ and all
$\eta\in\Lambda^k(\AnyBund,w_2)$ at $x$, we let
\begin{equation}\label{BilinearForm}
\langle\omega,\eta\rangle=\frac{1}{k!}\,\omega_{a_1\ldots a_k}{}^B
\eta^{a_1\ldots a_k}{}_B.
\end{equation}
Now again let $k\in\Integers_{\geq 0}$ and  $w_1,w_2\in\Reals$ be
given, but now suppose that $w_1+w_2=2(k+1)-n$.  The formal adjoint of
$d_{A}:\Omega^k(\AnyBund,w_1)\rightarrow\Omega^{k+1}(\AnyBund,w_1)$
with respect to $\BilinearMap$ is easily computed to be the operator
$
\delta_{A}:\Omega^{k+1}(\AnyBund,w_2)\rightarrow\Omega^{k}(\AnyBund,w_2-2)
$
given by
\begin{equation}\label{deltaBundle}
(\delta_{A}\eta)_{a_1\ldots
    a_k}{}^B=-\nd_b\eta^b{}_{a_1\ldots a_k}{}^B.
\end{equation}

Once again let $k\in\Integers_{\geq 0}$ be given.  For any
$w\in\Reals$, an obvious adaptation of the above discussion defines
our conventions for the exterior derivative
$d:\Omega^k(M,w)\rightarrow\Omega^{k+1}(M,w)$ and its formal adjoint
$\delta$.  By removing the subscript $A$ and the index $B$ from
\nn{dBundle} and \nn{deltaBundle}, we obtain symbolic formulae for
$(d\omega)_{a_1\ldots a_{k+1}}$ and $(\delta\eta)_{a_1\ldots a_k}$.

We will always assume that the connection $A$ is independent of any
choice of metric from the conformal class $\cc$ on $M$.  The coupled
connection $\nd$ will transform under conformal change of the metric
$\g$, however, and we will need certain conformal transformation rules
involving $\nd$.  Specifically, let $w\in\Reals$, a one-form $\eta$,
and a two-form $\omega$ be given.  Suppose that $\eta$ and $\omega$
have weight $w$ and that both take values in $\Reals$ or $\AnyBund$.
Then
\[
  \widehat{\nd}_a\eta_b=\nd_a\eta_b+(w-1)\Upsilon_a\eta_b-\Upsilon_b\eta_a
  +\bg_{ab}\Upsilon^{c}\eta_c\,,
\]
and
\begin{equation}\label{TwoFormTran}
\widehat{\nabla}_{a} \omega_{bc} = \nabla_{a} \omega_{bc} + (w-2)
\Upsilon_{a} \omega_{bc} - \Upsilon_{b} \omega_{ac} - \Upsilon_{c}
\omega_{ba} + \bg_{ab}\Upsilon^{d} \omega_{d c}+
\bg_{ac}\Upsilon^{d}\omega_{bd}.
\end{equation}
Both of these transformation rules follow from elementary computations
and the fact that $A$ is conformally invariant.  From \nn{deltaBundle}
and \nn{TwoFormTran}, it follows that
\begin{equation}\label{deltahatomega}
(\widehat{\delta}_A\omega)_c=(\delta_A\omega)_c+(4-n-w)\Upsilon^a\omega_{ac}.
\end{equation}

Our work will focus on vector bundles $\adGenBund$ which we will
define as follows.  Let $\pi: P \rightarrow M$ be an arbitrary
$G$-bundle, for some Lie group $G$. Any finite-dimensional $G$-module
$\mathbb{\GenBund}$ gives rise to an associated vector bundle
$\GenBund$ over $M$. Let $A$ denote a $G$-connection on the bundle
$P$.  The same notation will be used to denote the corresponding
associated linear connection on the associated vector bundle
$\GenBund$.  The adjoint representation induces a bundle of
Lie-algebras over $M$, and the derivative of the $G$-representation
acting on $\mathbb{\GenBund}$ then gives a vector bundle $\adGenBund$,
which is a sub-bundle of $\End(\GenBund) \cong \GenBund \otimes
\GenBund^{*}$.  Let $x\in M$ and elements $\omega$ and $\eta$ of
$\adGenBund$ at $x$ be given.  We will let $\omega\eta$
denote the composition $\omega\circ\eta$.  We may write $\omega\eta$
 in index notation as $\omega^B{}_D\eta^D{}_C$.  We will use this
 notation when we define a metric on $\adGenBund$ below.  Our work will
 \textit{not} require the existence of a metric on $\GenBund$.

We will work with operators acting on $\Omega^k(\adGenBund,w)$ for
 various $k\in\Integers_{\geq 0}$ and $w\in\Reals$.  To define these
 operators, let $k\in\Integers_{\geq 0}$ be given.  Let $\nd$ denote
 the connection
 $A$, and recall that $\nd$ determines a dual connection $\nd^*$ on
 $\GenBund^*$.  Specifically, define $\nd^*$ by the formula
 $(\nd^*_Xu)(v)=X(u(v))-u(\nd_Xv)$.  Here $X\in\Gamma(TM)$,
 $u\in\Gamma(\GenBund^*)$, and $v\in\Gamma(\GenBund)$.

By coupling $\nd$ and $\nd^*$, we determine a connection on
$\adGenBund$, and we in turn couple this connection with the
Levi-Civita connection on $M$ to obtain a connection on
$\Omega^{k}(\adGenBund,w)$.  Then \nn{dBundle} defines an exterior
derivative
$d_{A}:\Omega^{k}(\adGenBund,w)\rightarrow\Omega^{k+1}(\adGenBund,w)$
associated to $A$.  When we work with this $d_{A}$, we will append a
lower index $C$ to each side of \nn{dBundle}.

To give a formula for the formal adjoint of $d_A$, we will need a
metric on $\adGenBund$.  For any $x\in M$ and any $\omega$ and $\eta$
in $\adGenBund$ at $x$, let $h(\omega,\eta)=\omega^B{}_C\eta^C{}_B$.
Then by construction, $h$ is preserved by the connection $A$. For the
general theory, and for the purposes of deriving equations, we will
assume that $G$ is such that $h$ is non-degenerate, so $h$ is a bundle
metric on $\adGenBund$. (For example this is always the case if $G$ is
semisimple.  Then $h$ is a multiple of the Killing form.)  Now let
$k\in\Integers_{\geq 0}$ be given.  Let $\AnyBund=\adGenBund$, and
let $\langle\cdot,\cdot\rangle$ be the bilinear mapping described in
\nn{BilinearForm} and in the discussion preceding \nn{BilinearForm}.
In this context, we may rewrite \nn{BilinearForm} as
\[
\langle\omega,\eta\rangle=\frac{1}{k!}\,\omega_{a_1\ldots
  a_k}{}^B{}_C\,\eta^{a_1\ldots a_kC}{}_{B}.
\]
and \nn{deltaBundle} as
\begin{equation}\label{AdjointadV}
(\delta_A\eta)_{a_1\ldots
    a_k}{}^B{}_C=-\nd_b\eta^b{}_{a_1\ldots a_k}{}^B{}_C.
\end{equation}

\begin{remk}\label{sig}
Note that since the $h$ is defined by the canonical self-duality of
$\End (V)$, it is non-degenerate but with mixed signature. Indeed if, for example,
we use any choice of positive definite metric to identify $V$ with
$V^*$, and hence to identify $\End (V)$ with $\otimes^2 V$, then it is
easily seen that $h$ is positive definite on the symmetric part
$S^2V\subset \otimes^2 V$, and negative definite on the skew part
$\Lambda^2V\subset \otimes^2 V $. Of course the signature of $h$ is
independent of the choice made.

Although for the general theory we assume assume that $G$ is
such that $h$  is non-degenerate on $\adGenBund$, the equations that we
derive can be used more widely.
\end{remk}

We will sometimes consider the induced connections or twisted
exterior derivatives on $\GenBund$ and on $\adGenBund$ at the same
time, and we will use the same notations (such as $d_{A}$) for the versions of these
operators on both $\GenBund$ and $\adGenBund$.

The space of all connections on $\GenBund$ will be denoted
$\mathcal{A}$.  For any $A_1,\ A_2\in\mathcal{A}$, one can easily show
that $A_1-A_2\in\Omega^1(\adGenBund)$.  So once a connection
$A_{0}\in\mathcal{A}$ has been chosen, any other connection
$A\in\mathcal{A}$ can be written as $A_{0} + a$, for some $a \in
\Omega^{1}(\adGenBund)$.  Thus $\mathcal{A}$ is an affine space
modeled on the vector space
$\Omega^{1} (\adGenBund)$.  Suppose we choose a local frame for
$\GenBund$ and let $A_{0}=d$ be the trivial connection for this frame
(the exterior derivative on components).  Then for any
$A\in\mathcal{A}$, there is an $a\in\Omega^1(\adGenBund)$ such that $A
= d + a$.  This is the usual local expression for a linear connection.

Now let $A$, $A_0\in\mathcal{A}$ be given.  Then $A=A_0+a$ for some
$a\in\Omega^1(\adGenBund)$, as we noted above.  Next, let
$k\in\Integers_{\geq 0}$ be given,
and note that $A$ and $A_0$ induce connections on
$\TenBundZerok\otimes\End(\GenBund)$.  Let $A$ and $A_0$,
respectively, denote these induced connections.  Then for any
$\phi_{j_1\ldots
  j_k}{}^B{}_C\in\Gamma(\TenBundZerok\otimes\End(\GenBund))$,
elementary computations show that
\begin{equation}\label{ANaughtPlusa}
  A_i\phi_{j_1\ldots j_k}{}^B{}_C=
  (A_0)_i\phi_{j_1\ldots j_k}{}^B{}_C
  +a_i{}^B{}_E\phi_{j_1\ldots j_k}{}^E{}_C
  -\phi_{j_1\ldots j_k}{}^B{}_Ea_{i}{}^E{}_C.
\end{equation}
When we regard $A$ and $A_0$ as connections on
$\TenBundZerok\otimes\End(\GenBund)$, we will thus say that
$A=A_0+[a,\cdot]$.  In this context, we will also let $A_0+a$ denote
$A_0+[a,\cdot]$.  We do this, for example, in \nn{L9July21A}.
%
%
\subsection{Curvature}\label{CurvatureSect}
There is a (gauge) invariant section $F_{A}\in\Omega^{2}(\adGenBund)$ such that
$d_{A}^{2}S^B = F_{A}S^B$ for all $S^B\in \Gamma(\GenBund)$.  Thus
\begin{equation}\label{FCurveID}
(\nd_a\nd_b-\nd_b\nd_a)S^B = (F_A)_{ab}{}^B{}_ES^E
\end{equation}
for all $S^B\in\Gamma(\GenBund)$.  Here $\nd$ is again the coupled
connection. Because the Levi-Civita connection is torsion free,
$F_{A}$ is the usual \emph{curvature} of the connection $A$.
Similarly, there is an $F_A^*\in\Omega^2(\adGenBund^*)$ such that
$(\nd_a\nd_b-\nd_b\nd_a)S_B = (F_A^*)_{ab}{}^C{}_BS_C$ for all
$S_B\in\Gamma(\GenBund^*)$.  One can use the duality of $\GenBund$ and
$\GenBund^*$ to show that
$(F_A^*)_{ab}{}^C{}_B=-(F_A)_{ab}{}^{C}{}_B$.  Thus
\begin{equation}\label{FCurveID2}
(\nd_a\nd_b-\nd_b\nd_a)S_B = -(F_A)_{ab}{}^E{}_BS_E
\end{equation}
for all $S_B\in\Gamma(V^*)$.  The curvature $F_A$ satisfies the
Bianchi identity $d_AF_A=0$, as noted in \cite{DonaldsonKron}.
%
%
\subsection{The coupled $Q_{2}^A$-operator}\label{QTwoSect}
As we noted in Section~\ref{IntroSect}, the 
$Q$-curvature  is defined in \cite{BrO,BrFunctDet,Brsharp}.  In
\cite{BransonGoverGen,BransonGover}, the first author and Branson
define and discuss a family of linear operators $Q_{k}$, for $n$ even
and $k=0,1,2,\ldots,n/2+1$, mapping $\Omega^{k}(M)$ to
$\Omega^k(M,2k-n)$.  The $Q_k$-operators generalise Branson's
$Q$-curvature in several ways.  Branson's original $Q$-curvature is
equal to $Q_{0}1$.  We now restrict to manifolds of dimension $n=6$.
The case of interest to us then is the case $k=2$, where the operator
$Q_{2}$ is given as follows:
\begin{equation}\label{QTwoFormula}
Q_{2} \omega := d \delta \omega - 4 P \# \omega + 2 J \omega.
\end{equation}

In \cite{BransonGover}, the above-mentioned authors also discuss a
generalised family of operators $Q_k$ which act on bundle-valued
$k$-forms. We need only the $k=2$ case here, and we construct such
generalised operators by replacing the exterior derivative and its
adjoint in \nn{QTwoFormula} with the twisted exterior derivative
$d_{A}$ and its adjoint $\delta_{A}$, respectively.  We do this in the
next definition.  In this definition and in much of what follows, we
suppress bundle indices.  We will do so without further comment, since
the context should make this clear.
%
%
\begin{defn}\label{QTwoDef}
Let a pseudo-Riemannian manifold $(M,g)$ of dimension $n=6$, a vector bundle $\AnyBund$ over
$M$, and a bundle metric $h$ on $\AnyBund$
 be given.  Also let $A$ be a conformally invariant connection
 on $\AnyBund$ that preserves $h$.  Let the operator
$Q_{2}^A:\Omega^2(\AnyBund)\rightarrow\Omega^2(\AnyBund,-2)$ be
defined as follows:
\[
Q^A_{2} \omega= d_{A} \delta_{A} \omega - 4 P \# \omega+ 2 J \omega.  
\]
\end{defn}

The following proposition describes the behaviour of $Q_2^A$ under
conformal rescaling.
\begin{prop}\label{Q2rescaling}
Let $n$, $\AnyBund$, $A$, and $\omega$ be as in
Definition~\ref{QTwoDef}.  Suppose also that $\omega$ is $d_A$-closed.
Then under a conformal rescaling of the metric $\g$ on $M$, the
operator $Q_{2}^A$ obeys the following transformation rule:
\begin{equation}\label{QTwoTransformLaw}
  \widehat{Q}^A_{2} \omega
  =
  Q^A_{2}\omega + 2 \delta_A d_A ( \Upsilon \omega).
\end{equation}
\end{prop}
\noindent Proposition~\ref{Q2rescaling} is straightforward to verify,
and is a special case of Theorem~5.3 of \cite{BransonGover}.  Readers
familiar with the Branson's original $Q$-curvature will recognise the
similarity between \nn{QTwoTransformLaw} and the corresponding
transformation rule for the original $Q$-curvature.  We will apply
$Q^A_2$ to the curvature of a connection in Section~\ref{ActionSect},
below.  We will apply Proposition~\ref{Q2rescaling} in
sections~\ref{ActionSect} and \ref{ComplexSect}. Finally in this
Section, we present here a result that we will need in those Sections.
\begin{prop}\label{QFSA}
The operator $Q_2^A$ is formally self-adjoint.
\end{prop}
\begin{proof}
For any $\eta_{ab},\,\omega_{ab}\in\Omega^2(\AnyBund)$,
\[
P_a{}^{c}\eta^{ab}\omega_{cb}+P_b{}^c\eta^{ab}\omega_{ac}
=
P_c{}^a\eta^{cb}\omega_{ab}+P_c{}^b\eta^{ac}\omega_{ab}
=
P_a{}^c\eta_{cb}\omega^{ab}+P_b{}^c\eta_{ac}\omega^{ab},
\]
so the result follows from Definition~\ref{QTwoDef}.
\end{proof}
%
%
\section{The conformally invariant action in six
  dimensions}\label{ActionSect}
Let $M$ be a closed 6-manifold equipped with a conformal class $\cc$
of pseudo-Riemannian metrics and let notation be as in
Section~\ref{ConnectionSect} above.  So $A$ denotes a connection on a
vector bundle $\GenBund$.  Note that
\begin{equation}\label{LagD}
  \left< F_{A}, Q_{2}^{A} F_{A}
  \right>\in\Gamma(\ce[-6]),
\end{equation}
and define the integration of densities as
in Section~\ref{IntegrationSection}. This allows us to define an
action on $A$ as follows.
\begin{defn}\label{BigSADef}
Suppose that $n=6$ and that $M$ is closed.  Define an action
$\Action(A)$ by
\begin{equation}\label{SADef}
\Action(A) := \int_{M} \left< F_{A}, Q_{2}^{A} F_{A} \right> d \mu.
\end{equation}
\end{defn}
\begin{prop}\label{SACoInv}
The action $\Action(A)$ is invariant under conformal
change of the metric $\g$ on $M$.
\end{prop}
\begin{proof}
Note that $d_AF_A=0$, by the Bianchi identity.  Thus
\[
\int_M\langle F_A,\widehat{Q}_2^AF_A\rangle\,d\mu=
\int_M \langle F_{A} ,Q_2^AF_A+ 2 \delta_{A} d_{A}(\Upsilon
F_{A})\rangle\,d\mu
\]
by Proposition~\ref{Q2rescaling}.  But $M$ has empty boundary and $A$
preserves the pairing.  Thus
\[
\int_M
\langle F_{A} ,2\delta_Ad_{A}(\Upsilon F_{A})\rangle
\,d\mu
=
\int_M
\langle d_{A} F_{A},2 d_{A} (\Upsilon F_{A}) \rangle\,
d\mu
=
0,
\]
and the proposition follows.
\end{proof}
%
%
\subsection{Euler-Lagrange equations for
  $\Action$}\label{ELSection}
We will now derive the Euler-Lagrange equations for the conformally
invariant action $\Action$.  Here and below, $[\delta_AF_A,F_A]$ will
denote $(\delta_AF_A)F_A-F_A\delta_A F_A$.  This notation includes an
implicit contraction of the free lowercase index of $\delta_A F_A$
with the first lowercase index of the other $F_A$ within the brackets.
\begin{prop}\label{EulerLagrange}
The Euler-Lagrange equations for the action $\Action$ are
\[
\delta_{A}Q^{A}_2 F_{A} - [ \delta_{A} F_{A}, F_{A} ]=0.
\]
\end{prop}
\begin{proof}
For $a \in \Omega^{1} (\adGenBund)$ and $t \in \mathbb{R}$, let $A+ t
a$ be a path of connections on $\GenBund$ starting at $A$. Then, since
$(F_{A + ta})_{ij}{}^B{}_C = (F_{A})_{ij}{}^B{}_C + t(d_{A}
a)_{ij}{}^B{}_C
+ t^{2}(a\wedge a)_{ij}{}^B{}_C$, it follows that
\[
\left.\frac{d}{dt}
  \Action(A + t a)\,
\right|_{t=0}
=
\int_{M}
\left(
  \left< d_{A} a, Q^{A}_2 F_{A} \right>
  + \left< F_{A}, Q^{A}_2 d_{A} a \right>
  + \langle F_{A} , \dot{Q}^{A}_2 F_{A} \rangle
\right)d\mu.
\]
Here $\dot{Q}_{2}^{A}$ denotes the linearisation of $Q_2^{A}$.  Thus
by Proposition~\ref{QFSA},
\begin{equation}\label{actionTwo}
\left.\frac{d}{dt}
  \Action(A + t a)\,
\right|_{t=0}
=
\int_M\left(
2 \langle d_{A} a , Q_2^{A} F_{A} \rangle + \langle F_{A},
\dot{Q}_2^{A} F_{A} \rangle\right)d\mu.
\end{equation}
Consider the integral of $\langle F_A,\dot{Q}_2^{A}F_A\rangle$.  Only
the leading term of $Q_2^{A}$ will contribute to this integral.  We
find that
\begin{equation}\label{L9July21A}
\frac{d}{d t}\left.\int_M\langle F_{A}, d_{A+ta} \delta_{A+ta}
F_{A} \rangle\,d\mu\,\right|_{t=0}
=
\int_M2 \langle
\dot{\delta}_{A} F_{A}, \delta_{A} F_{A} \rangle\,d\mu.
\end{equation}
By using \nn{AdjointadV} and \nn{ANaughtPlusa}, one can easily show
that
$
\delta_{A + t a} F_{A} = \delta_{A} F_{A} - t [a, F_{A}]
$.
Here the bracket includes an implicit contraction between the
lowercase index of $a$ and the first lowercase index of $F_{A}$.  Thus
$\dot{\delta}_{A} F_{A} = - [ a, F_{A} ]$.  From this it follows that
\begin{equation}\label{EL-BBB}
\int_M\langle
F_A,\dot{Q}_2^A F_A
\rangle\,d\mu
=
-\int_M
2  \langle [ a, F_{A} ], \delta_{A} F_{A}\rangle\,
d\mu.
\end{equation}
Now again let $\nd$ denote the coupling of $A$ and the Levi-Civita
connection, and note that
\begin{equation}\label{EL-AAA}
\begin{array}{ll}
\lefteqn{-2\langle[a,F_A],\delta_AF_A\rangle=
  -2(a_c{}^G{}_HF^c{}_d{}^H{}_I
  -F_{cd}{}^G{}_Ha^{cH}{}_I)
  (-\nd_eF^{edI}{}_G)=
}\vspace{0.8ex}
\\
&
2(a_c{}^G{}_HF^c{}_d{}^H{}_I\nd_e F^{edI}{}_G
-F_{cd}{}^G{}_H a^{cH}{}_I\nd_eF^{edI}{}_G)=
\vspace{0.8ex}
\\
&
2(a_c{}^G{}_HF_d{}^{cH}{}_I(-1)\nd_eF^{edI}{}_G
-a^{cH}{}_IF_{dc}{}^G{}_H(-1)\nd_eF^{edI}{}_G)=
\vspace{0.8ex}
\\
&
-2(-a_c{}^G{}_HF_d{}^{cH}{}_I(\delta_A F)^{dI}{}_G
+a^{cH}{}_I(\delta_A F)^{dI}{}_GF_{dc}{}^G{}_H)
=
-2\langle a,[\delta_A F_A,F_A]\rangle .
\end{array}
\end{equation}
In most of the expressions in \nn{EL-AAA}, we have omitted the
subscript $A$ from $F$, but we have included the indices of $F_A$.
From \nn{actionTwo}, \nn{EL-BBB}, and \nn{EL-AAA}, it follows that
\[
\begin{array}{ll}
\lefteqn{\frac{d}{dt} S(A+ta)\,\bigg|_{t=0}
=
\int_M\left(
2\langle d_{A} a, Q^{A}_2 F_{A} \rangle - 2  \langle a , [\delta_{A}
  F_{A}, F_{A}] \rangle
\right)
d\mu=}\vspace{1.5ex}
\\
&
\ds
\int_M 2
\left\langle a,
  \left(\delta_A Q^{A}_2 F_{A}-[\delta_{A} F_{A}, F_{A}]\right)
\right\rangle
d\mu,
\hspace{24ex}
\end{array}
\]
and the proposition follows.
\end{proof}
%
%
\section{A higher-order analogue of the Yang-Mills equations}\label{ComplexSect}

We continue the setting and notation of the previous section.
\subsection{Definition of $\mathfrak{D}$}\label{DDefSect}
The results in the previous section motivate the following definition:
\begin{defn}\label{DDef}
Define an operator $\mathfrak{D}: \mathcal{A} \rightarrow \Omega^{1}
(\End(\GenBund))$ as follows:
\[
\CurlyD A:= \delta_{A}Q_2^{A} F_{A} - [\delta_{A} F_{A}, F_{A} ].
\]
\end{defn}
\noindent
By comparing Proposition~\ref{EulerLagrange} to the usual derivation
of the Yang-Mills equations as the Euler-Lagrange equations for the
functional $||F_{A} ||^{2}$, one can view the equation $\CurlyD A=0$
as a conformally invariant analogue, for dimension 6, of the usual source-free
Yang-Mills equations in four dimensions. 

%
%
\subsection{Conformal invariance of $\CurlyD$}\label{DInvarianceSect}
The fact that the action $\Action(A)$ is conformally invariant
suggests that the Euler-Lagrange equations $\mathfrak{D}A=0$ are also
conformally invariant. This is verified explicitly in the following
proposition:
\begin{prop}\label{conformalinv}
On psuedo-Riemannian 6-manifolds, the operator $\CurlyD$ is invariant
under conformal change of the metric $\g$ on $M$.
\end{prop}
\begin{proof}
  In this proof, we often omit indices.  Note first that $F_A$ and
  $Q_2^AF_A$ carry weights $0$ and $-2$, respectively.  Thus
\[
\widehat{\delta}_A \widehat{Q}_2^A F_A
=
\widehat{\delta}_A (Q_2^AF_A + 2 \delta_A d_A (\Upsilon F_A))
= \delta_AQ_2^AF_A + 2 \delta_A^{2} d_A( \Upsilon F_A),
\]
by \nn{deltahatomega} and \nn{QTwoTransformLaw} since $n=6$.
Similarly,
\begin{equation}\label{L16June20}
  -[ \widehat{\delta}_A
    F_A, F_A]
  =
  -[ \delta_A F_A, F_A] + 2 [\iota(d \Upsilon) F_A, F_A],
\end{equation}
where $\iota(d\Upsilon)$ denotes interior multiplication by
$\Upsilon^a$.  In the second term on the right-hand side of
\nn{L16June20}, the bracket notation includes a contraction of the
free lowercase index of $\iota(d\Upsilon)F_A$ with the first lowercase
index of the second $F_A$ within the brackets.  To complete the proof,
it suffices to show that
\begin{equation}\label{SufficesEquation}
\delta^{2}_A d_A (\Upsilon F_A)
+
[\iota(d\Upsilon) F_A, F_A]=0.
\end{equation}
To do this, we begin by recalling that $d_AF_A=0$.  From this it
follows that
\begin{equation}\label{lastterm}
(\delta^2_Ad_A(\Upsilon F))_c
=
2\nabla^{[e} \nabla^{d]} ( \Upsilon_{d} F_{ec})
+\nabla^{[e} \nabla^{d]} (\Upsilon_{c} F_{de}). 
\end{equation}
Here and for the rest of this proof, we omit the subscript $A$ from
$F$.  Again, $\nd$ denotes the coupling of $A$ and the Levi-Civita
connection.  By \nn{RiemannConvention}, \nn{FCurveID}, and
\nn{FCurveID2}, the right-hand side of \nn{lastterm} is equal to
\begin{equation}\label{Equation14June20}
\begin{array}{l}
  R_{ed}{}^d{}_j\Upsilon^jF^{e}{}_c{}^G{}_H
  +
  R_{ed}{}^e{}_j\Upsilon^dF^j{}_c{}^G{}_H
  -R_{ed}{}^j{}_c\Upsilon^dF^e{}_j{}^G{}_H
  \vspace{0.5ex}
  \\
  -\frac{1}{2}R_{ed}{}^j{}_c\Upsilon_jF^{deG}{}_H
  +\frac{1}{2}R_{ed}{}^d{}_j\Upsilon_cF^{jeG}{}_H
  +\frac{1}{2}R_{ed}{}^e{}_j\Upsilon_cF^{djG}{}_H
  \vspace{0.5ex}
  \\
  +F_{ed}{}^G{}_I\Upsilon^dF^{e}{}_c{}^I{}_H
  -F_{ed}{}^I{}_H\Upsilon^dF^e{}_c{}^G{}_I
  +\frac{1}{2}F_{ed}{}^G{}_I\Upsilon_cF^{deI}{}_H
  -\frac{1}{2}F_{ed}{}^I{}_H\Upsilon_cF^{deG}{}_I.
\end{array}
\end{equation}
Here we have included the indices associated to $\GenBund$ and
$\GenBund^*$.  The sum of the first six terms of \nn{Equation14June20}
is zero, and the sum of the last two terms of \nn{Equation14June20} is
also zero.  Thus
\[
(\delta_A^2d_A(\Upsilon F))_c{}^G{}_H
=
-(\iota(d\Upsilon)F)_e{}^G{}_IF^e{}_c{}^I{}_H
+F^e{}_c{}^G{}_I\,(\iota(d\Upsilon)F)_e{}^I{}_H
=
-[\iota(d\Upsilon)F,F]_c{}^G{}_H.
\]
This establishes \nn{SufficesEquation}.
\end{proof}
%

%
%
\section{Applications}\label{ApplicationsSection}
In this section we specialise to the case that $A$ is the conformal
tractor connection of \cite{Thomas}. As usual we work on a closed
conformal 6-manifold $(M,\cc )$. In this case, in terms of a metric
$\g\in \cc$, our Lagrangian density \nn{LagD} is
\begin{equation}\label{conformal}
\left< F_{A}, Q_{2}^{A} F_{A} \right>
=
4A_{abc}\nd^bP^{ac}
-JC_{abcd}C^{abcd}
+4C_{abcd}\nd^d\nd^bP^{ac}
+4P_{ab}C^a{}_{cde}C^{bcde}.
\end{equation}
If we work modulo divergences, we may simplify the right-hand side of
\nn{conformal} further.  (By \textit{divergences}, we mean terms of
the form $\nd_iT^i$, where $T^i\in\Gamma(TM\otimes \mathcal{E}[-6])$.)  We find that
$$
\Action{(A)}
=
\int_M(
8A_{abc}\nd^cP^{ab}
-JC_{abcd}C^{abcd}
+4P_{ab}C^a{}_{cde}C^{bcde}
)\,d\mu .
$$
For this action we will compute the Euler-Lagrange equations of
Proposition \ref{EulerLagrange}.  The main result is Theorem
\ref{ObstructionTheorem} below which shows that these equations
recover the condition of $\ObTenSix =0$, where $ \ObTenSix$ is the
Fefferman-Graham obstruction tensor of \cite{FG1,FG2} in dimension
$6$. Thus Proposition \ref{EulerLagrange}, and the operator of
Definition~\ref{DDef} so defined, provide a new perspective and way of
constructing a symbolic formula for the Fefferman-Graham obstruction
tensor $\ObTenSix$.  This symbolic formula will express $\ObTenSix$ in
terms of the Levi-Civita connection and tensors associated to the
Riemannian curvature tensor.

%
%
\subsection{The standard tractor bundle}\label{TBundSect}
We begin by reviewing some basic facts concerning the standard tractor
bundle and the conformal tractor connection.  We are interested in
pseudo-Riemannian metrics in dimension $n=6$, but our discussion of
tractors in this subsection treats pseudo-Riemannian metrics in
general dimensions $n$.  We refer the reader to \cite{Thomas} and
\cite{GP-CMP} for further information.

Let a manifold $M$ of dimension $n\geq 3$ and a conformal class $\cc$
of pseudo-Riemannian metrics on $M$ be given.  Let $(p,q)$ denote the
signature of the metrics in $\cc$.  The conformal class $\cc$
determines a certain $(n+2)$-dimensional vector bundle over $M$
commonly known as the (standard) \textit{tractor bundle}.  We let
$\TBund$ (or $\TBundB$ in abstract index notation) denote this bundle.
The conformal class also determines a connection on $\TBund$ which we
call the \textit{tractor connection}.  Let $\nd^{\TBund}$ or simply
$\nd$ denote this connection.  The symbol $\nd$ will also denote the
coupled Levi-Civita tractor connection.  This connection acts on
powers of $\TBund$ in the obvious way.

A choice of metric $\g\in\cc$ determines a vector
bundle isomorphism
\[
I_{\gsub}:
\dbun[1]
\oplus(\TenBundZeroOne\otimes\dbun[1])
\oplus\dbun[-1]\rightarrow\TBundB.
\]
The chosen metric $\g$ also determines algebraic splitting operators
\[
Y^B\in\Gamma(\TBundB\otimes\dbun[-1]),
\hspace{1ex}
Z^{Bc}\in\Gamma(\TBundB\otimes\TenBundOneZero\otimes\dbun[-1]),
\hspace{1ex}
X^B\in\Gamma(\TBundB\otimes\dbun[1])
\]
such that for all
$(\sigma,\mu_c,\rho)\in\dbun[1]\oplus(\TenBundZeroOne\otimes\dbun[1])
\otimes\dbun[-1]$,
\[
I_{\gsub}(\sigma,\mu_c,\rho)=\sigma Y^B+\mu_cZ^{Bc}+\rho X^B.
\]
The splitting operator $X^B$ is conformally invariant.  When the
coupled Levi-Civita tractor connection acts on the three splitting
operators, it obeys the following rules:
\begin{equation}\label{connids}
\nd_aY^B=P_{ac}Z^{Bc},\hspace{3ex}
\nd_aZ^B{}_c=-P_{ac}X^B-\bg_{ac}Y^B,\hspace{3ex}
\nd_aX^B=Z^B{}_a.
\end{equation}

The conformal structure $\cc$ also determines a metric $\h$ on
$\TBundB$ which has signature $(p+1,q+1)$.  We say that $\h$ is the
\textit{tractor metric}.  The tractor connection $\nd$ preserves $\h$.
We may use $\h$ to raise and lower tractor indices, and this operation
commutes with the action of $\nd$.  One can show that $Y_BX^B=1$ and
$Z_{Ba}Z^{B}{}_c=\bg_{ac}$.  All other contractions of tractor indices
involving pairs of the splitting operators $Y^B$, $Z^{Bc}$, and
$X^{B}$ are zero.

Let $\Omega_{ab}{}^D{}_E$ denote the curvature of $\nd^{\TBund}$.
Thus $(\nd_a\nd_b-\nd_b\nd_a)V^D=\Omega_{ab}{}^D{}_EV^E$ for all
$V^D\in\TBund$.  We say that $\Omega_{ab}{}^D{}_E$ is the
\textit{tractor curvature}.  One can show that
\begin{equation}\label{PlainOmega}
\Omega_{ab}{}^D{}_E = \tensor{Z}{^D}^{c} \tensor{Z}{_{E}}^{e}
C_{abce}
- X^{D}Z_{E}{}^{e} \tensor{A}{_{eab}}
+ X_{E}Z^{D}{}^{e} \tensor{A}{_{eab}}\,,
\end{equation}
and
\begin{equation}\label{MinusdeltaOmega}
\nabla^{a}\Omega_{ac}{}^D{}_E
= (n-4) Z^{Dd}\tensor{Z}{_E^e} \tensor{A}{_{cde}}
- \tensor{X}{^D} \tensor{Z}{_{E}^e}\tensor{B}{_{ec}}
+ \tensor{X}{_E} Z^{De}\tensor{B}{_{ec}}.
\end{equation}

Observe that, according to \nn{deltahatomega}, in dimension $n=4$ the
operator $\delta_{\nabla^\mathcal{T}}$ is conformally invariant on
weight 0 tractor-bundle-valued 2-forms.  So
$\nabla^{a}\Omega_{ac}{}^D{}_E$ is conformally invariant in dimension
four. On the other hand, in any dimension
\begin{equation}\label{XZminusZX}
X_E Z^{Dd}-X^D Z_{E}{}^d   
\end{equation}
is conformally invariant, by the conformal transformation rules for
$Z^{Dd}$ and $Z_E{}^d$ given in equation (4) of \cite{GP-CMP}. Thus
\nn{XZminusZX} provides, up to a constant factor, the standard and
well-known conformally invariant injection of 1-forms into the adjoint
tractor bundle $\Lambda^2\mathcal{T}$. Thus in dimension four, the
display \nn{MinusdeltaOmega} is recovering the result \nn{YM-Bach}
mentioned in the Introduction.

%
\subsection{The obstruction tensor in dimension 6}\label{FirstSubsect}
In this subsection, we will describe symbolic computations which
establish the following theorem:
\begin{thm}\label{ObstructionTheorem}
Suppose that $n=6$.  Then
\[
(\CurlyD\nd^{\TBund})_c{}^D{}_E
=
(X_E Z^{Dd}-X^D Z_{E}{}^d)16\ObTenSixcd\,,
\]
and hence
\begin{equation}\label{SecondObTenForm}
\ObTenSixce
=
\frac{1}{32}(Y^EZ_{De}-Y_DZ^{E}{}_e)(\CurlyD\nd^{\TBund})_c{}^D{}_E.
\end{equation}
\end{thm}
Here and throughout our work, our scaling convention for $\ObTenSix$
is the same as in \cite{GP-Obstrn}.  One can expand the right-hand
side of \nn{SecondObTenForm} and compute a symbolic formula which
expresses $\ObTenSixce$ in terms of the Weyl tensor, the Schouten
tensor, and $J$.

\begin{proof}[Proof of Theorem~\ref{ObstructionTheorem}] We begin by applying
Definition~\ref{DDef}.  In this definition, we replace $A$ with
$\nd^{\TBund}$ and $F$ with the tractor curvature $\Omega$.  We work
in dimension 6.  We conclude that
\begin{equation}\label{FullExpression}
\begin{array}{ll}
\lefteqn{(\CurlyD\nd^{\TBund})_c{}^D{}_E=
(\delta d\delta \Omega)_c{}^D{}_E
- 4(\delta (P \# \Omega))_c{}^D{}_E
+ 2 (\delta(J \Omega))_c{}^D{}_E}\vspace{0.5ex}
\\
&
-(\delta \Omega)_b{}^D{}_G\Omega^{b}{}_c{}^G{}_E
+\Omega^b{}_{c}{}^D{}_G(\delta\Omega)_b{}^{G}{}_E
.
\hspace*{19ex}
\end{array}
\end{equation}
In \nn{FullExpression} and for the rest of this subsection, we omit
the subscript $\nd^{\TBund}$ from $\delta$ and $d$.  Our plan will be
to consider the various summands on the right-hand side of
\nn{FullExpression} separately.

Consider the first summand, namely $(\delta d\delta\Omega)_c{}^D{}_E$.
From \nn{dBundle}, \nn{AdjointadV}, and \nn{MinusdeltaOmega}, it
follows that
\begin{equation}\label{d27July19a}
\begin{array}{ll}
\lefteqn{(d \delta \Omega)_{a_1a_2}{}^D{}_E
  =}\vspace{0.5ex}\\
&
-4\nabla_{a_{1}} ( \tensor{Z}{^D^d} \tensor{Z}{_E^e} \tensor{A}{_{ a_{2} d e}}) 
+ 2\nabla_{a_{1}} (\tensor{X}{^{D}} \tensor{Z}{_{E}^e}\tensor{B}{_{e a_{2}}})
- 2\nabla_{a_{1}} (\tensor{X}{_{E}} \tensor{Z}{^{D}^e}\tensor{B}{_{e a_{2}}})
.
\end{array}
\end{equation}
On the right-hand side of this display, we antisymmetrise the indices
$a_{1}$ and $a_{2}$.  By \nn{connids} and \nn{d27July19a},
\begin{equation}\label{d27July19b}
\begin{array}{ll}
  \lefteqn{(d \delta \Omega)_{a_1a_2}{}^D{}_E =}
\vspace{0.5ex}
\\
  &4 \tensor{P}{_{a_{1}}^d} \tensor{X}{^D} \tensor{Z}{_E^e}
\tensor{A}{_{a_{2} d e}}
+ 4 \tensor{Y}{^D}
  \tensor{Z}{_E^e} \tensor{A}{_{a_{2}a_1 e}}
+
  4 \tensor{Z}{^{Dd}} \tensor{P}{_{a_{1}}^e} \tensor{X}{_E}
  \tensor{A}{_{a_{2} d e}}
\vspace{0.5ex}
\\
&
+4 \tensor{Z}{^{Dd}}\tensor{Y}{_E}
\tensor{A}{_{a_{2} d a_1}}
- 4 \tensor{Z}{^{Dd}} \tensor{Z}{_E^e}\nabla_{a_{1}}
\tensor{A}{_{a_{2} d e}}
+ 2Z^{D}{}_{a_{1}}
  \tensor{Z}{_{E}^e} \tensor{B}{_{e a_{2}}}
\vspace{0.5ex}
\\
&
+  2\tensor{X}{^{D}} \tensor{Z}{_{E}^e} 
\nabla_{a_{1}} \tensor{B}{_{e a_{2}}}
- 2\tensor{Z}{_{E a_{1}}}
  \tensor{Z}{^{De}} \tensor{B}{_{e a_{2}}}
- 2\tensor{X}{_{E}} \tensor{Z}{^{De}} 
\nabla_{a_{1}} \tensor{B}{_{e a_{2}}}.
\end{array}
\end{equation}
Here we again antisymmetrise the indices $a_1$ and $a_2$.  From
\nn{AdjointadV}, \nn{connids}, \nn{d27July19b}, and the fact that
$A_{abc}$ is trace-free, it follows that $(\delta
d\delta\Omega)_{a_2}{}^D{}_E$ is given by the formula in
Figure~\ref{deltaddeltaOmega}.
\begin{figure}
\begin{align}\nonumber 
&
4\tensor{X}{^D}\tensor{X}{_E}
\tensor{P}{_{a_{1}}^d}\tensor{P}{^{a_{1}e}}
\tensor{A}{_{a_{2} d e}}
+4\tensor{X}{^D} 
\tensor{Y}{_E}\tensor{P}{_{a_{1}}^d}\tensor{A}{_{a_{2} d}^{a_1}}
- 4\tensor{X}{^D} \tensor{Z}{_E^e}
\tensor{P}{_{a_{1}}^d} \nabla^{a_{1}} \tensor{A}{_{a_{2} d e}}
\\ \nonumber 
&
-4\tensor{X}{^D} \tensor{Z}{_E^e} (\nabla^{a_{1}} \tensor{P}{_{a_{1}}^d} )
\tensor{A}{_{a_{2} d e} }
-4 \tensor{Z}{^{Da_{1}}} \tensor{Z}{_E^e} 
\tensor{P}{_{a_{1}}^d} \tensor{A}{_{a_{2} d e}} 
+ 4\tensor{Y}{^D} \tensor{X}{_E}\tensor{P}{^{a_{1} e}}
\tensor{A}{_{a_{2} a_{1} e} }
\\ \nonumber 
&
-4\tensor{Y}{^D} \tensor{Z}{_E^e} \nabla^{a_{1}} 
\tensor{A}{_{a_{2} a_{1} e} }
-4 \tensor{Z}{^{Dd}} \tensor{Z}{_E^e} \tensor{P}{^{ a_{1} }_d}
\tensor{A}{_{a_{2} a_{1} e}}
+ 4\tensor{X}{^D}\tensor{X}{_E}
\tensor{P}{^{ a_{1} d } }\tensor{P}{_{ a_{1} }^e}
\tensor{A}{_{a_{2} d e} }
\\ \nonumber 
&
+ 4 \tensor{Y}{^D}\tensor{X}{_E}\tensor{P}{_{ a_{1} }^e} 
A_{a_{2}}{}^{a_1}{}_e
-4\tensor{Z}{^{Dd}} \tensor{X}{_E} \tensor{P}{_{a_{1}}^e} 
\nabla^{a_{1}} \tensor{A}{_{a_{2} d e}}
-4 \tensor{Z}{^{Dd}} \tensor{X}{_E} (\nabla^{a_{1}} 
\tensor{P}{_{a_{1}}^e}) \tensor{A}{_{a_{2} d e} }
\\ \nonumber 
&
-4\tensor{Z}{^{Dd}}\tensor{Z}{_E^{ a_{1} }}\tensor{P}{_{a_{1}}^e}
\tensor{A}{_{a_{2} d e}}
+4 \tensor{X}{^D}\tensor{Y}{_E}\tensor{P}{^{a_{1} d}}
\tensor{A}{_{a_{2} d a_{1}}}
-4\tensor{Z}{^{Dd}} \tensor{Y}{_E} 
\nabla^{a_{1}} \tensor{A}{_{a_{2} d a_{1}}}
\\ \nonumber 
&
-4\tensor{Z}{^{Dd}} \tensor{Z}{_E^e}\tensor{P}{^{a_{1}}_e}
\tensor{A}{_{a_{2} d a_{1}}}
-4 \tensor{X}{^D} \tensor{Z}{_E^e}\tensor{P}{^{a_{1} d}}\nabla_{a_{1}}
\tensor{A}{_{a_{2} d
    e}}
-4
\tensor{Y}{^D} \tensor{Z}{_E^e}\nabla_{a_{1}} A_{a_{2}}{}^{a_1}{}_e
\\ \nonumber 
&
-4 \tensor{Z}{^{Dd}} \tensor{X}{_E}\tensor{P}{^{a_{1} e}}
\nabla_{a_{1}} 
\tensor{A}{_{a_{2} d e}}
-4 \tensor{Z}{^{Dd}} \tensor{Y}{_E}\nabla_{a_{1}} 
\tensor{A}{_{a_{2} d}^{a_1}}
+4 \tensor{Z}{^{Dd}} \tensor{Z}{_E^e} 
\nabla^{a_{1}} \nabla_{a_{1}} \tensor{A}{_{a_{2} d e}}
\\  \nonumber 
&
+2\tensor{X}{^{D}}\tensor{Z}{_{E}^e}
\tensor{P}{^{a_{1}}_{a_{1}}}\tensor{B}{_{e a_{2}}}
+
2\tensor{Y}{^{D}} \tensor{Z}{_{E}^e}\tensor{\g}{^{a_{1}}_{a_{1}}}
\tensor{B}{_{e a_{2}}}
+2\tensor{Z}{^D_{a_{1}}}\tensor{X}{_{E}}
\tensor{P}{^{a_{1} e}}\tensor{B}{_{e a_{2} }}
+2\tensor{Z}{^D_{a_{1}}} 
\tensor{Y}{_{E}} \tensor{B}{^{a_1}_{a_{2} }}
\\ \nonumber 
&
-2 \tensor{Z}{^{D}_{a_1}} \tensor{Z}{_{E}^e} 
\nabla^{a_{1}} \tensor{B}{_{e a_{2}}}
+2 \tensor{X}{^{D}}\tensor{Y}{_{E}} 
\nabla_{a_{1}} B^{a_1}{}_{a_{2}}
-2 \tensor{X}{^{D}} \tensor{Z}{_{E}^e} 
\nabla^{a_{1}} \nabla_{a_{1}} \tensor{B}{_{e a_{2}}}
\\ \nonumber 
&
-2 \tensor{Z}{^{Da_{1}}} \tensor{Z}{_{E}^e} 
\nabla_{a_{1}} \tensor{B}{_{e a_{2}}}
-2 \tensor{Z}{_{E a_{1}}}\tensor{X}{^{D}}
\tensor{P}{^{a_{1} e}}\tensor{B}{_{e a_{2} }}
-2 \tensor{Z}{_{E a_{1}}} 
\tensor{Y}{^{D}} \tensor{B}{^{a_1}_{a_{2} }}
-2\tensor{X}{_{E}}\tensor{Z}{^{De}}
\tensor{P}{^{a_{1}}_{a_{1}}}\tensor{B}{_{e a_{2}}}
\\ \nonumber 
&
-
2\tensor{Y}{_{E}} \tensor{Z}{^{De}}\tensor{\g}{^{a_{1}}_{a_{1}}}
\tensor{B}{_{e a_{2}}}
+2\tensor{Z}{_{E a_1}} \tensor{Z}{^{D}^e} 
\nabla^{a_{1}} \tensor{B}{_{e a_{2}}}
-2\tensor{X}{_{E}}\tensor{Y}{^{D}} 
\nabla_{a_{1}} B^{a_1}{}_{a_{2}}
\\ \nonumber 
&
+2\tensor{X}{_{E}} \tensor{Z}{^{D}^e} 
\nabla^{a_{1}} \nabla_{a_{1}} \tensor{B}{_{e a_{2}}}
+2\tensor{Z}{_{E}^{a_{1}}} \tensor{Z}{^{De}} 
\nabla_{a_{1}} \tensor{B}{_{e a_{2}}}. 
\end{align}
\caption{Symbolic formula for $(\delta d\delta\Omega)_{a_2}{}^D{}_E$ in
  dimension 6}\label{deltaddeltaOmega}
\end{figure}
In this formula, we antisymmetrise the \textit{lower} indices $a_1$
and $a_2$.  The upper index $a_1$ does \textit{not} participate in the
antisymmetrisation.  The index $a_2$ is a free index.  One may replace
this index with the index $c$ to conform with \nn{FullExpression}.

Now consider the second summand on the right-hand side of
\nn{FullExpression}, namely $-4(\delta(P\#\Omega))_c{}^D{}_E$.  By
\nn{PlainOmega},

\begin{align}
\nonumber
-4(P \# \Omega)_{bc}{}^D{}_E =\,&
-4\tensor{P}{_b^k} \tensor{Z}{^{Dd}}\tensor{Z}{_E^e} 
\tensor{C}{_{kcde}}
+4 \tensor{P}{_b^k} \tensor{X}{^D} \tensor{Z}{_E^e} 
\tensor{A}{_{ekc}}
-4 \tensor{P}{_b^k} \tensor{X}{_E} \tensor{Z}{^{Dd}}
\tensor{A}{_{dkc}}
\\ \nonumber
&
-4\tensor{P}{_c^k} \tensor{Z}{^{Dd}} \tensor{Z}{_E^e} \tensor{C}{_{bkde}} 
+4\tensor{P}{_c^k} \tensor{X}{^D} \tensor{Z}{_E^e} \tensor{A}{_{ebk}}
-4\tensor{P}{_c^k} \tensor{X}{_E} \tensor{Z}{^{Dd}}
\tensor{A}{_{dbk}}.
\end{align}
From this and from \nn{AdjointadV}, it follows that
\[
\begin{array}{ll}
%
%
  \lefteqn{-4(\delta ( P \# \Omega))_c{}^D{}_E=}\vspace{0.5ex}
\\
%
%
&
4\nabla^{b}( \tensor{P}{_b^k} \tensor{Z}{^{Dd}} \tensor{Z}{_E^e}
\tensor{C}{_{kcde}} )
-4\nabla^{b} ( \tensor{P}{_b^k} \tensor{X}{^D}\tensor{Z}{_E^e}
\tensor{A}{_{ekc}} )
+4\nabla^{b}( \tensor{P}{_b^k} \tensor{X}{_E} \tensor{Z}{^{Dd}}
\tensor{A}{_{dkc}} )\vspace{0.5ex}
\\
%
%
&
+4\nabla^{b} ( \tensor{P}{_c^k} \tensor{Z}{^{Dd}}
\tensor{Z}{_E^e} \tensor{C}{_{bkde}} )
-4\nabla^{b} ( \tensor{P}{_c^k} \tensor{X}{^D} \tensor{Z}{_E^e}
\tensor{A}{_{ebk}} )
+4\nabla^{b} ( \tensor{P}{_c^k} \tensor{X}{_E} \tensor{Z}{^{Dd}}
\tensor{A}{_{dbk}}).
\end{array}
\]
Thus from \nn{connids} and the fact that the Cotton tensor and the
Weyl tensor are trace-free, it follows that $-4(\delta ( P \#
\Omega))_c{}^D{}_E$ is given by the formula in
Figure~\ref{delPHashOmOne}.
\begin{figure}
\begin{align}\label{delPhash}
\nonumber 
&
-4\tensor{X}{^D}\tensor{Z}{_E^e}\tensor{P}{_b^k}\tensor{P}{^{bd}}
\tensor{C}{_{kcde}}
-4\tensor{Y}{^D}\tensor{Z}{_E^e}\tensor{P}{^d^k}
\tensor{C}{_{kcde}}
-4\tensor{Z}{^D^d} \tensor{X}{_E}\tensor{P}{_b^k}\tensor{P}{^{be}}
\tensor{C}{_{kcde}}
\\ \nonumber 
&
-4\tensor{Z}{^D^d}\tensor{Y}{_E}\tensor{P}{^e^k}\tensor{C}{_{kcde}}
+4 \tensor{Z}{^{Dd}} \tensor{Z}{_E^e}
(\nabla^{b} \tensor{P}{_b^k})\tensor{C}{_{kcde}}
+4\tensor{Z}{^D^d} \tensor{Z}{_E^e}
\tensor{P}{_b^k}\nabla^{b} \tensor{C}{_{kcde}}
\\ \nonumber 
&
+4\tensor{X}{^D}\tensor{X}{_E}
\tensor{P}{_b^k}\tensor{P}{^{be}}\tensor{A}{_{ekc}}
+4\tensor{X}{^D}\tensor{Y}{_E}
\tensor{P}{^e^k}\tensor{A}{_{ekc}}
-4\tensor{X}{^D}\tensor{Z}{_E^e}
\tensor{P}{_b^k}\nabla^{b}\tensor{A}{_{ekc}}
\\ \nonumber 
&
-4\tensor{X}{^D} \tensor{Z}{_E^e}(\nabla^{b} \tensor{P}{_b^k}) 
\tensor{A}{_{ekc}}
-4\tensor{Z}{^D^b} \tensor{Z}{_E^e}
\tensor{P}{_b^k}\tensor{A}{_{ekc}} 
-4
\tensor{X}{_E}\tensor{X}{^D}
\tensor{P}{_b^k}\tensor{P}{^{bd}}\tensor{A}{_{dkc}}
\\ \nonumber 
&
-4\tensor{X}{_E}\tensor{Y}{^D}
\tensor{P}{^d^k}\tensor{A}{_{dkc}}
+4\tensor{X}{_E}\tensor{Z}{^D^d}
\tensor{P}{_b^k}\nabla^{b} \tensor{A}{_{dkc}}
+4 \tensor{X}{_E} \tensor{Z}{^D^d}
(\nabla^{b}\tensor{P}{_b^k})\tensor{A}{_{dkc}}
\\ \nonumber 
&
+4\tensor{Z}{_E^b} \tensor{Z}{^D^d}
\tensor{P}{_b^k}\tensor{A}{_{dkc}}
-4\tensor{X}{^D}\tensor{Z}{_E^e}
\tensor{P}{_c^k}\tensor{P}{^{bd}}
\tensor{C}{_{bkde}}
-4\tensor{Z}{^D^d}\tensor{X}{_E}
\tensor{P}{_c^k}\tensor{P}{^{be}}\tensor{C}{_{bkde}}
\\ \nonumber 
&
+4\tensor{Z}{^D^d} \tensor{Z}{_E^e}
(\nabla^{b} \tensor{P}{_c^k})\tensor{C}{_{bkde}}
+4\tensor{Z}{^D^d}\tensor{Z}{_E^e}
\tensor{P}{_c^k}\nabla^{b}\tensor{C}{_{bkde}}
+4\tensor{X}{^D}\tensor{X}{_E}
\tensor{P}{_c^k}\tensor{P}{^{be}}\tensor{A}{_{ebk}}
\\ \nonumber 
&
-4\tensor{X}{^D}\tensor{Z}{_E^e}
\tensor{P}{_c^k}\nabla^{b}\tensor{A}{_{ebk}}
-4\tensor{X}{^D}\tensor{Z}{_E^e}
(\nabla^{b}\tensor{P}{_c^k})\tensor{A}{_{ebk}}
-4\tensor{Z}{^D^b}\tensor{Z}{_E^e}
\tensor{P}{_c^k}\tensor{A}{_{ebk}} 
\\ \nonumber 
&
-4\tensor{X}{_E}\tensor{X}{^D} 
\tensor{P}{_c^k}\tensor{P}{^{bd}}\tensor{A}{_{dbk}}
+4\tensor{X}{_E}\tensor{Z}{^D^d}
\tensor{P}{_c^k}\nabla^{b}\tensor{A}{_{dbk}}
+4\tensor{X}{_E} \tensor{Z}{^D^d}
(\nabla^{b} \tensor{P}{_c^k})\tensor{A}{_{dbk}}
\\ \nonumber 
&
+4\tensor{Z}{_E^b}\tensor{Z}{^D^d}
\tensor{P}{_c^k}\tensor{A}{_{dbk}}.
\end{align}
\caption{Symbolic formula for $-4(\delta ( P \# \Omega))_c{}^D{}_E$ in
  dimension 6}
\label{delPHashOmOne}
\end{figure}

We now consider the third summand on the right-hand side of
\nn{FullExpression}, namely $2(\delta(J\Omega))_c{}^D{}_E$.
By \nn{AdjointadV} and \nn{PlainOmega}, 
\[
\begin{array}{ll}
\lefteqn{2(\delta( J \Omega))_c{}^D{}_E
=-2\nabla^{a} ( J \Omega_{ac}{}^D{}_E)=}\vspace{0.5ex}
\\
&
-2\nabla^{a} (
J \tensor{Z}{^D^d} \tensor{Z}{_E^e} \tensor{C}{_{acde}}
-J\tensor{X}{^D}\tensor{Z}{_E^e} \tensor{A}{_{eac}}
+J\tensor{X}{_E}\tensor{Z}{^D^e} \tensor{A}{_{eac}}
).
\end{array}
\]
From \nn{connids} and from the fact that the Cotton and Weyl tensors
are trace-free, it follows that
\begin{equation}\label{ThirdSummand}
\begin{array}{ll}
%
%
\lefteqn{2(\delta(J\Omega))_c{}^D{}_E=}\vspace{0.5ex}
\\ 
&
-2\tensor{Z}{^D^d} \tensor{Z}{_E^e}( \nabla^{a} J)\tensor{C}{_{acde}} 
+2\tensor{X}{^D}\tensor{Z}{_E^e}J \tensor{P}{^{ad}}\tensor{C}{_{acde}}
-2\tensor{Z}{^D^d} \tensor{Z}{_E^e}J
\nabla^{a} \tensor{C}{_{acde}}
\vspace{0.5ex}
\\ 
&
+2\tensor{X}{_E}\tensor{Z}{^D^d} J\tensor{P}{^{ae}}
\tensor{C}{_{acde}}
+2\tensor{X}{^{D}}\tensor{Z}{_{E}^e}(\nabla^{a} J)
\tensor{A}{_{eac}} 
-2\tensor{X}{_{E}}\tensor{Z}{^D^e}(\nabla^{a} J)
\tensor{A}{_{eac}}
\vspace{0.5ex}
\\ 
&
+2\tensor{Z}{^D^a} \tensor{Z}{_{E}^e}J\tensor{A}{_{eac}}
-2\tensor{Z}{_{E}^a} \tensor{Z}{^D^e}J\tensor{A}{_{eac}}
+2\tensor{X}{^{D}}\tensor{Z}{_{E}^e}J\nabla^{a} \tensor{A}{_{eac}}
\vspace{0.5ex}
\\ 
&
-2\tensor{X}{_{E}}\tensor{Z}{^D^e}J\nabla^{a} \tensor{A}{_{eac}}.
\end{array}
\end{equation}

The fourth summand is $-(\delta\Omega)_b{}^D{}_G\Omega^b{}_c{}^G{}_E$.
By \nn{AdjointadV}, \nn{PlainOmega} and \nn{MinusdeltaOmega},
\[
\begin{array}{ll}
\lefteqn{-(\delta\Omega)_b{}^D{}_G\Omega^b{}_c{}^G{}_E=}\vspace{0.5ex}
\\
&
(
2 \tensor{Z}{^{Dd}} \tensor{Z}{_G^k} \tensor{A}{_{bdk}} 
-\tensor{X}{^D} \tensor{Z}{_G^k} \tensor{B}{_{kb}}
+\tensor{X}{_G}
\tensor{Z}{^{Dk}} \tensor{B}{_{kb}}
) \times\vspace{0.5ex}
\\
&
( \tensor{Z}{^{Gi}} \tensor{Z}{_E^e} \tensor{C}{^b_{cie}} 
- \tensor{X}{^G} \tensor{Z}{_E^e} \tensor{A}{_e^b_c} + \tensor{X}{_E} 
\tensor{Z}{^{Ge}} \tensor{A}{_e^b_c}
).
\end{array}
\]
From this and from the rules for the tractor metric as applied to the
splitting operators $Y^B$, $Z^{Bc}$, and $X^B$, it follows that
\begin{equation}\label{FourthSummand}
\begin{array}{ll}
\lefteqn{-(\delta\Omega)_b{}^D{}_G\Omega^b{}_c{}^G{}_E=}\vspace{0.5ex}
\\
&
2\tensor{Z}{^{Dd}} \tensor{Z}{_E^e} \tensor{A}{_{bd}^i} 
\tensor{C}{^b_{cie}}
+2\tensor{Z}{^{Dd}} \tensor{X}{_E} 
\tensor{A}{_{bd}^e} \tensor{A}{_e^b_c}
-\tensor{X}{^D} \tensor{Z}{_E^e} \tensor{B}{^i_b}
\tensor{C}{^b_{cie}}
\\
&
-\tensor{X}{^D} \tensor{X}{_E} \tensor{B}{^e_b} \tensor{A}{_e^b_c}.
\end{array}
\end{equation}

The fifth summand is $\Omega^b{}_c{}^D{}_G(\delta\Omega)_b{}^G{}_E$.  By
\nn{AdjointadV}, \nn{PlainOmega} and \nn{MinusdeltaOmega},
\[
\begin{array}{ll}
\lefteqn{\Omega^b{}_c{}^D{}_G(\delta\Omega)_b{}^G{}_E=}\vspace{0.5ex}
\\
&
(
\tensor{Z}{^{Dd}} \tensor{Z}{_G^k} \tensor{C}{^b_{cdk}} 
- \tensor{X}{^D} \tensor{Z}{_G^k} \tensor{A}{_k^b_c}
+ \tensor{X}{_G}\tensor{Z}{^{Dk}}\tensor{A}{_k^b_c}  \vspace{0.5ex}
) \times
\\
&
(
-2\tensor{Z}{^{Gi}} \tensor{Z}{_E^e} \tensor{A}{_{bie}} 
+ \tensor{X}{^G} \tensor{Z}{_E^e} \tensor{B}{_{eb}}
- \tensor{X}{_E}\tensor{Z}{^{Ge}} \tensor{B}{_{eb}}
).
\end{array}
\]
Thus
\begin{equation}\label{FifthSummand}
\begin{array}{ll}
\lefteqn{\Omega^b{}_c{}^D{}_G(\delta\Omega)_b{}^G{}_E=}\vspace{0.5ex}
\\
&
-2\tensor{Z}{^{Dd}} \tensor{Z}{_E^e} \tensor{C}{^b_{cd}^i}\tensor{A}{_{bie}}
-\tensor{Z}{^{Dd}} \tensor{X}{_E}\tensor{C}{^b_{cd}^e} \tensor{B}{_{eb}}
+2\tensor{X}{^D}\tensor{Z}{_E^e}\tensor{A}{^{ib}_c}\tensor{A}{_{bie}}
\vspace{0.5ex}
\\
&
+\tensor{X}{^D} \tensor{X}{_E} \tensor{A}{^{eb}_c} \tensor{B}{_{eb}}.
\end{array}
\end{equation}

Our next step will be to combine (implicitly) several of the symbolic
formulae that we computed above.  Note first, however, that in the
symbolic formula appearing in Figure~\ref{deltaddeltaOmega}, the index
$a_2$ is a free index.  \textit{After} we perform the
antisymmetrisation of the lower indices $a_1$ and $a_2$, we may
replace the index $a_2$ with the index $c$.  We can then combine the
symbolic formulae appearing in figures~\ref{deltaddeltaOmega} and
\ref{delPHashOmOne} together with the symbolic formulae appearing on
the right-hand sides of \nn{ThirdSummand}, \nn{FourthSummand}, and
\nn{FifthSummand}.  The result will be a symbolic formula for
$(\CurlyD\nd^{\TBund})_c{}^D{}_E$, which we will refer to as the
``total'' symbolic formula for $(\CurlyD\nd^{\TBund})_c{}^D{}_E$.  As
we will now show, one can simplify this total symbolic formula and
obtain a new total symbolic formula for
$(\CurlyD\nd^{\TBund})_c{}^D{}_E$ in which every term contains an
occurrence of either $X_EZ^{Dd}$ or $X^DZ_{E}{}^d$.

Note first that none of the terms in the total symbolic formula for
$(\CurlyD\nd^{\TBund})_c{}^D{}_E$ contain an occurrence of $Y_EY^D$.

Now consider the terms in the total symbolic formula that contain
$X^DY_E$ or $X_EY^D$.  The sum of these terms is
\begin{equation}\label{XYTermsOne}
\begin{array}{ll}
\lefteqn{
(X^DY_E-X_EY^D)(
2P_{a_1}{}^dA_{cd}{}^{a_1}
-2P_c{}^dA_{a_1d}{}^{a_1}
+2P^{a_1d}A_{cda_1}
}\vspace{0.5ex}
\\
&
-2P^{a_1d}A_{a_1dc}
+\nd_{a_1}B^{a_1}{}_c
-\nd_cB^{a_1}{}_{a_1}
+4P^{ek}A_{ekc}).\hspace{6.0ex}
\end{array}
\end{equation}
This follows by inspection and the fact that $A_{abc}$ is
antisymmetric in $b$ and $c$.  But $P_{ab}$ is symmetric, and
$A_{abc}$ and $B_{ab}$ are trace-free.  Thus \nn{XYTermsOne} is equal
to
\[
(X^DY_E-X_EY^D)(\nd_aB^a{}_c+2P^{ek}A_{ekc}).
\]
A short symbolic computation using \cite{Mathematica} and
\cite{RicciSoftware}, along with \nn{Bach}, \nn{CommuteCovD},
\nn{TraceNbP}, \nn{CottonTensor}, and \nn{WeylTensor}, shows that
$\nd_aB^{a}{}_c+2P^{ek}A_{ekc}=0$ in dimension 6.  Thus
\nn{XYTermsOne} is zero.

Next, consider the terms in the total symbolic formula for
$(\CurlyD\nd^{\TBund})_c{}^D{}_E$ in which $X^DX_E$ appears.  The sum
of these terms is
\begin{equation}\label{XXTermsOne}
\begin{array}{ll}
\lefteqn{
X^DX_E(
2P_{a_1}{}^dP^{a_1e}A_{cde}
-2P_c{}^dP^{a_1e}A_{a_1de}
+2P^{a_1d}P_{a_1}{}^eA_{cde}
-2P^{a_1d}P_{c}{}^eA_{a_1de}
}\vspace{0.5ex}
\\
&
+4P_{b}{}^kP^{be}A_{ekc}
-4P_b{}^kP^{bd}A_{dkc}
+4P_{c}{}^kP^{be}A_{ebk}
-4P_{c}{}^kP^{bd}A_{dbk}
-B^{e}{}_bA_{e}{}^b{}_c
\hspace{1.0ex}\vspace{0.5ex}
\\
&
+A^{eb}{}_cB_{eb}).
\end{array}
\end{equation}
Note that $P_{a_1}{}^dP^{a_1e}$ is symmetric in $d$ and $e$ and that
$A_{cde}$ is antisymmetric in $d$ and $e$.  Thus in the parenthesised
expression in \nn{XXTermsOne}, terms~1 and 3 are zero.  Since
$A_{a_1de}$ is antisymmetric in $d$ and $e$, it follows that terms~2
and 4 cancel.  Terms~5 and 6 cancel, terms~7 and 8 cancel, and terms~9
and 10 cancel.  Thus \nn{XXTermsOne} is zero.

We now consider the terms in the total symbolic formula for
$(\CurlyD\nd^{\TBund})_c{}^D{}_E$ in which $Y^DZ_{E}{}^i$,
$Y_EZ^{Di}$, $Y^DZ_{Ei}$, or $Y_EZ^{D}{}_i$ appears.  Here $i$ is a
dummy index.  The total symbolic formula may use some index other than
$i$ here.  We will also consider the terms of the total symbolic
formula in which $Y^DZ_{Ec}$ or $Y_EZ^{D}{}_c$ appears.  The sum of
the terms under consideration is as follows:
\begin{equation}\label{YZTermsOne}
\begin{array}{l}
(Y^DZ_E{}^i-Y_EZ^{Di})(
-2\nd^{a_1}A_{ca_1i}
+2\nd^{a_1}A_{a_1ci}
-2\nd_{a_1}A_{c}{}^{a_1}{}_i
+2\nd_cA_{a_1}{}^{a_1}{}_i\vspace{0.5ex}
\\
\hspace*{3.5ex}
+\g^{a_1}{}_{a_1}B_{ic}
-\g^{a_1}{}_cB_{ia_1}
-B_{ic}
-4P^{dk}C_{kcdi})\vspace{0.5ex}
\\
+(Y^DZ_{Ec}-Y_EZ^{D}{}_c)B^{a_1}{}_{a_1}\,.
\end{array}
\end{equation}
This follows from inspection of the total symbolic formula, from the
fact that $A_{abc}$ is antisymmetric in $b$ and $c$, and from the fact
that $C_{abcd}$ is antisymmetric in $c$ and $d$.  Now recall that the
Cotton tensor is trace-free and the Bach tensor is symmetric and trace
free.  Since $n=6$, it follows that \nn{YZTermsOne} is equal to
\begin{equation}\label{YZTermsTwo}
\begin{array}{l}
(Y^DZ_E{}^i-Y_EZ^{Di})(4B_{ci}
-4(\nd^{a_1}A_{ca_1i}+P^{kd}C_{kcdi})
+2\nd^{a_1}A_{a_1ci}).
\end{array}
\end{equation}
It is well-known that $\nd_aA^a{}_{bc}=0$.  (See formula~(2.7) of
\cite{Gover-Nurowski}, for example.)  Thus \nn{YZTermsTwo} is zero, by
\nn{Bach}, and hence \nn{YZTermsOne} is zero as well.

Now consider the terms of the total symbolic formula for
$(\CurlyD\nd^{\TBund})_{c}{}^D{}_E$ in which $Z^{Di}Z_E{}^k$,
$Z^D{}_cZ_{E}{}^e$, or $Z^{De}Z_{Ec}$ appears.  Here $i$, $k$, and $e$
are dummy indices; in the total symbolic formula, any of these indices
may appear as a lower index.  The sum of the terms under consideration
is given by the symbolic expression on the left-hand side of the
equation in Figure~\ref{ZZTermsOne}.
%
%
\begin{figure}
\[
\begin{array}{l}
Z^{Di}Z_E{}^k(
-2P_i{}^dA_{cdk}
+2P_c{}^dA_{idk}
-2P^{a_1}{}_iA_{ca_1k}
+2P^{a_1}{}_iA_{a_1ck}
-2P_k{}^eA_{cie}
+2P_c{}^eA_{kie}
\vspace{0.5ex}
\\
\hspace*{3.5ex}
-2P^{a_1}{}_kA_{cia_1}
+2P^{a_1}{}_kA_{a_1ic}
+2\nd^{a_1}\nd_{a_1}A_{cik}
-2\nd^{a_1}\nd_cA_{a_1ik}
-\nd_iB_{kc}
-\nd_iB_{kc}\vspace{0.5ex}
\\
\hspace*{3.5ex}
+\nd_cB_{ki}
+\nd_kB_{ic}
+\nd_kB_{ic}
-\nd_cB_{ik}
+4(\nd^bP_{b}{}^a)C_{acik}
+4P_{b}{}^a\nd^bC_{acik}\vspace{0.5ex}
\\
\hspace*{3.5ex}
-4P_i{}^aA_{kac}
+4P_k{}^aA_{iac}
+4(\nd^bP_c{}^a)C_{baik}
+4P_c{}^a\nd^bC_{baik}
-4P_c{}^aA_{kia}\vspace{0.5ex}
\\
\hspace{3.5ex}
+4P_c{}^aA_{ika}
-2(\nd^aJ)C_{acik}
-2J\nd^aC_{acik}
+2JA_{kic}
-2JA_{ikc}
+2A_{bi}{}^aC^{b}{}_{cak}\vspace{0.5ex}
\\
\hspace*{3.5ex}
-2C^b{}_{ci}{}^aA_{bak}
)\vspace{0.5ex}
\\
+Z^D{}_cZ_E{}^e\nd^{a_1}B_{ea_1}-Z^{De} Z_{Ec}\nd^{a_1}B_{ea_1}\vspace{0.5ex}
\\
\hspace*{3.5ex}=\vspace{0.5ex}
\\
Z^{Di}Z_{E}{}^k(-4P_i{}^{d}A_{cdk}
+2P_c{}^aA_{ika}
+2P^{a_1}{}_iA_{a_1ck}
-4P_{k}{}^{e}A_{cie}
-2P_c{}^aA_{kia}\vspace{0.5ex}
\\
\hspace*{3.5ex}
+2P^{a_1}{}_kA_{a_1ic}
+2\nd^{a_1}\nd_{a_1}A_{cik}
-2\nd^{a_1}\nd_{c}A_{a_1ik}
-2\nd_iB_{kc}
+2\nd_kB_{ic}\vspace{0.5ex}
\\
\hspace*{3.5ex}
+2(\nd^aJ)C_{acik}
+4P_{b}{}^a\nd^bC_{acik}
-4P_i{}^aA_{kac}
+4P_k{}^aA_{iac}
+4(\nd^{b}P_{c}{}^a)C_{baik}\vspace{0.5ex}
\\
\hspace*{3.5ex}
+4P_c{}^a\nd^bC_{baik}
-2J\nd^aC_{acik}
+2JA_{kic}
-2JA_{ikc}
+2A_{bi}{}^aC^b{}_{cak}
-2C^{b}{}_{ci}{}^aA_{bak})\vspace{0.5ex}
\\
+Z^D{}_cZ_E{}^e\nd^{a_1}B_{ea_1}-Z^{De}Z_{Ec}\nd^{a_1}B_{ea_1}
\end{array}
\]
\caption{Sum of the terms of the total symbolic formula for
  $(\CurlyD\nd^{\TBund})_c{}^D{}_E$ in which $Z^{Di}Z_E{}^k$,
  $Z^D{}_cZ_{E}{}^e$, or $Z^{De}Z_{Ec}$ appears.}\label{ZZTermsOne}
\end{figure}
%
%
  This follows from direct inspection of the total symbolic formula
  for $(\CurlyD\nd^{\TBund})_{c}{}^D{}_E$.  The equation in the figure
  follows from simplification of its left-hand side.  To perform this
  simplification, one may use the symmetry of $B_{ab}$, the
  antisymmetry of $A_{abc}$ in $b$ and $c$, and \nn{TraceNbP}.  By
  using \cite{Mathematica} and \cite{RicciSoftware}, one can show that
  the right-hand side of the equation is zero.  To do this, one may
  begin with the right-hand side of the equation and use \nn{Bach} and
  \nn{CottonTensor} and to express $A_{abc}$ and $B_{ab}$ in terms of
  $P_{ab}$ and $C_{abcd}$.  If one then applies \nn{CommuteCovD} in a
  suitable way, all derivatives of order greater than 1 cancel.  One
  may then use the Bianchi identities and \nn{WeylTensor} to show that
  the remaining terms cancel.

Finally, consider the terms of the total symbolic formula for
$(\CurlyD\nd^{\TBund})_c{}^D{}_E$ in which $X_EZ^D{}^i$, $X^DZ_E{}^i$,
$X_EZ^D{}_c$, or $X^DZ_E{}_c$ occurs.  Here $i$ is a dummy index, as
before.  In the total symbolic formula, this $i$ may appear as a lower
index.  The sum of the terms under consideration is given by the
leftmost member of the equation in Figure~\ref{XZTermsOne}.
%
%
\begin{figure}
\[
\begin{array}{ll}
%
%
X_{E}Z^{Di}(
-2P_{a_1}{}^e\nd^{a_1}A_{cie}
+2P_c{}^e\nd^{a_1}A_{a_1ie}
-2(\nd^{a_1}P_{a_1}{}^e)A_{cie}
+2(\nd^{a_1}P_c{}^e)A_{a_1ie}
\vspace{0.5ex}
\\
\hspace*{3.5ex}
-2P^{a_1e}\nd_{a_1}A_{cie}
+2P^{a_1e}\nd_{c}A_{a_1ie}
+P_i{}^eB_{ec}
-P^{a_1}{}_{a_1}B_{ic}
+P^{a_1}{}_cB_{ia_1}
+\nd^{a_1}\nd_{a_1}B_{ic}\vspace{0.5ex}
\\
\hspace*{3.5ex}
-\nd^{a_1}\nd_{c}B_{ia_1}
-4P_{b}{}^kP^{be}C_{kcie}
+4P_b{}^k\nd^bA_{ikc}
+4(\nd^bP_{b}{}^k)A_{ikc}
-4P_c{}^kP^{be}C_{bkie}\vspace{0.5ex}
\\
\hspace*{3.5ex}
+4P_{c}{}^k\nd^bA_{ibk}
+4(\nd^bP_{c}{}^k)A_{ibk}
+2JP^{ae}C_{acie}
-2(\nd^aJ)A_{iac}
-2J\nd^aA_{iac}\vspace{0.5ex}
\\
\hspace*{3.5ex}
+2A_{bi}{}^eA_{e}{}^b{}_c
-C^b{}_{ci}{}^eB_{eb})\vspace{0.5ex}
\\
\hspace*{7.0ex}-Z^D{}_cX_EP^{a_1e}B_{ea_1}\vspace{0.5ex}
\\
+X^DZ_E{}^i(
-2P_{a_1}\!{}^d\nd^{a_1}A_{cdi}
+2P_c{}^d\nd^{a_1}A_{a_1di}
-2(\nd^{a_1}P_{a_1}{}^d)A_{cdi}
+2(\nd^{a_1}P_c{}^d)A_{a_1di}\vspace{0.5ex}
\\
\hspace*{3.5ex}
-2P^{a_1d}\nd_{a_1}A_{cdi}
+2P^{a_1d}\nd_cA_{a_1di}
+P^{a_1}{}_{a_1}B_{ic}
-P^{a_1}{}_cB_{ia_1}
-\nd^{a_1}\nd_{a_1}B_{ic}\vspace{0.5ex}
\\
\hspace*{3.5ex}
+\nd^{a_1}\nd_{c}B_{ia_1}
-P_{i}{}^eB_{ec}
-4P_{b}{}^kP^{bd}C_{kcdi}
-4P_{b}{}^k\nd^{b}A_{ikc}
-4(\nd^bP_b{}^k)A_{ikc}\vspace{0.5ex}
\\
\hspace*{3.5ex}
-4P_c{}^kP^{bd}C_{bkdi}-4P_c{}^k\nd^{b}A_{ibk}
-4(\nd^bP_c{}^k)A_{ibk}
+2JP^{ad}C_{acdi}
+2(\nd^aJ)A_{iac}\vspace{0.5ex}
\\
\hspace*{3.5ex}
+2J\nd^aA_{iac}
-B^{a}{}_bC^b{}_{cai}
+2A^{ab}{}_cA_{bai})\vspace{0.5ex}
\\
\hspace*{7.0ex}
+Z_{Ec}X^{D}P^{a_1e}B_{ea_1}
\vspace{0.5ex}
\\
%
%
\hspace*{3.5ex}\vspace{0.5ex}=
\\
(X_{E}Z^{Di}-X^DZ_{E}{}^i)(
-2P_{a_1}{}^e\nd^{a_1}A_{cie}
+2P_c{}^e\nd^{a_1}A_{a_1ie}
-2(\nd^{a_1}P_{a_1}{}^e)A_{cie}
\vspace{0.5ex}
\\
\hspace*{3.5ex}
+2(\nd^{a_1}P_c{}^e)A_{a_1ie}
-2P^{a_1e}\nd_{a_1}A_{cie}
+2P^{a_1e}\nd_{c}A_{a_1ie}
+P_i{}^eB_{ec}
-P^{a_1}{}_{a_1}B_{ic}
\vspace{0.5ex}
\\
\hspace*{3.5ex}
+P^{a_1}{}_cB_{ia_1}
+\nd^{a_1}\nd_{a_1}B_{ic}
-\nd^{a_1}\nd_{c}B_{ia_1}
-4P_{b}{}^kP^{be}C_{kcie}
+4P_b{}^k\nd^bA_{ikc}
\vspace{0.5ex}
\\
\hspace*{3.5ex}
+4(\nd^bP_{b}{}^k)A_{ikc}
-4P_c{}^kP^{be}C_{bkie}\vspace{0.5ex}
+4P_{c}{}^k\nd^bA_{ibk}
+4(\nd^bP_{c}{}^k)A_{ibk}
+2JP^{ae}C_{acie}
\\
\hspace*{3.5ex}
-2(\nd^aJ)A_{iac}
-2J\nd^aA_{iac}\vspace{0.5ex}
+2A_{bi}{}^eA_{e}{}^b{}_c
-C^b{}_{ci}{}^eB_{eb}
-\bg_{ic}P^{a_1e}B_{ea_1}
)
%
%
\\
\hspace*{3.5ex}\vspace{0.5ex}=
\\
(X_{E}Z^{Di}-X^DZ_{E}{}^i)(
-8P_{ab}\nd^aA_{(ic)}{}^b
+2P_c{}^e\nd^{a_1}A_{a_1ie}
-4A_{(ic)a}\nd^aJ
\vspace{0.5ex}
\\
\hspace*{3.5ex}
+2(\nd^{a_1}P_c{}^e)A_{a_1ie}
+2P^{a_1e}\nd_{c}A_{a_1ie}
+P_i{}^eB_{ec}
-3JB_{ic}
+5P_c{}^kB_{ik}
+\Delta B_{ic}
\vspace{0.5ex}
\\
\hspace*{3.5ex}
-\nd^{a_1}\nd_{c}B_{ia_1}
+4P_{b}{}^kP^{be}C_{ieck}
+4(\nd^bP_{c}{}^k)A_{ibk}
-2A_{bie}A^{e}{}_c{}^b
+B_{eb}C_i{}^e{}_c{}^b
\vspace{0.5ex}
\\
\hspace*{3.5ex}
-\bg_{ic}P^{a_1e}B_{ea_1}
)
\end{array}
\]
\caption{Sum of the terms in the total symbolic formula for
 $(\CurlyD\nd^{\TBund})_c{}^D{}_E$ in which $X_EZ^D{}^i$, $X^DZ_E{}^i$,
 $X_EZ^D{}_c$, or $X^DZ_E{}_c$ occurs.}\label{XZTermsOne}
\end{figure}
This follows from direct inspection of the total symbolic formula for
$(\CurlyD\nd^{\TBund})_c{}^D{}_E$.  The equation in the figure follows
from the symmetries of $A_{abc}$ and $C_{abcd}$ and from \nn{Bach} and
\nn{TraceNbP}.  In dimension $n=6$, the rightmost member of this
equation is equal to $(X_EZ^{Dd}-X^DZ_E{}^d)16\ObTenSixcd$.  This
follows from a direct computation using \cite{Mathematica} and
\cite{RicciSoftware}, together with \nn{Bach}, \nn{CommuteCovD}, \nn{TraceNbP},
\nn{CottonTensor}, and \nn{WeylTensor}.  The proof of
Theorem~\ref{ObstructionTheorem} is thus complete.
\end{proof}

%
%
\subsection{A comment on the Fefferman-Graham invariant in dimension
  6} \label{FGIRemark}

Fefferman and Graham introduced an interesting natural scalar
conformal invariant $\FGI$ in Proposition~3.4 of \cite{FG2} (see also
\cite{CG-amb}). This invariant has conformal weight $-6$ and so on
closed conformal 6-manifolds,
\begin{equation}\label{FGI-int}
\int_M  \FGI~d\mu
\end{equation}
is a well-defined global conformal invariant. Moreover at leading
order \FGI ~ is quadratic in the jets of the metric. Thus \nn{FGI-int}
is an important action to consider for metric variations. In fact, it
is closely linked to the specialisation to the tractor connection of
our action \nn{SADef} as follows.

Let $A$ denote the tractor connection,
and let $F_A$ denote the tractor curvature as given in
\nn{PlainOmega}.  Then in dimension $n=6$,
\begin{equation}\label{FGICompare}
  \begin{array}{ll}
    \lefteqn{\FGI=}\vspace{0.0ex}
  \\
  &
  2\,\langle F_A,Q_2^AF_A\rangle
  +8C_{abcd}C^a{}_e{}^c{}_iC^{bedi}
  -4C_{abcd}C^a{}_e{}^c{}_iC^{bide}
  +\mbox{{\upshape(}\textit{divergences}\hspace{0.25ex}{\upshape)}}.
  \end{array}
\end{equation}
Here $\langle\cdot,\cdot\rangle$ and $Q_2^A$ are as in
Definition~\ref{BigSADef}.  The equality in \nn{FGICompare} follows
from a short symbolic computation using the computer algebra system
\textit{Mathematica}, \cite{Mathematica}, together with Lee's tensor
calculus software package \textit{Ricci}, \cite{RicciSoftware}.

In even dimensions, elementary representation theory shows that the
Weyl tensor (and its irreducible parts in dimension 4) and the
Fefferman-Graham obstruction tensor are the only conformal invariants
that at leading order are linear in the jets of the metric \cite{Graham-Hirachi}. 
Using this and Theorem \ref{ObstructionTheorem}, it is not difficult
to conclude that {\em with respect to metric variations} both
\nn{SADef} (with $A$ the conformal tractor connection) and
\nn{FGI-int} must yield $\ObTenSixce$ as their functional gradient, up
to conformally invariant lower-order terms.

%
%

\end{document}